%% file: main.tex
\documentclass[11pt]{article}

\usepackage{macros}

\newcommand{\Nbr}{\mathrm{Nbr}}
\newcommand{\LNbr}{\mathrm{LNbr}}
\newcommand{\RNbr}{\mathrm{RNbr}}
\newcommand{\RED}{\mathsf{RED}}
\newcommand{\Livne}{Livn{\'e}\xspace}
\newcommand{\GLSU}{G_L[S,U]} % previously this was \Gamma_S, which is a bit confusing.
\newcommand{\Out}{\mathrm{Out}}

\newif\ifnotes
\notestrue

\title{Explicit Lossless Vertex Expanders}

\author{Jun-Ting Hsieh\thanks{Carnegie Mellon University. \texttt{juntingh@cs.cmu.edu}. Supported by NSF CAREER Award \#2047933.
This work was done while the author was visiting MIT.}
\and Alexander Lubotzky\thanks{Weizmann Institute, Rehovot, Israel.
\texttt{alex.lubotzky@mail.huji.ac.il}.
Supported by the European Research Council (ERC) under the European Union’s Horizon 2020 (N. 882751), and the research grant from the Center for New Scientists at the Weizmann Institute of Science. This work was done while the author was visiting the department of mathematics at MIT, whose hospitality and support is gratefully acknowledged.}
\and Sidhanth Mohanty\thanks{MIT. \texttt{sidm@mit.edu}. Supported by NSF Award DMS-2022448.}
\and Assaf Reiner\thanks{Hebrew University of Jerusalem, Jerusalem, Israel.
\texttt{assaf.reiner@mail.huji.ac.il}.
Supported by the European Research Council (ERC) under the European Union’s Horizon 2020 (N. 882751).}
\and Rachel Yun Zhang\thanks{MIT.  \texttt{rachelyz@mit.edu}. Supported by NSF Graduate Research Fellowship 2141064. Supported in part by DARPA under Agreement No. HR00112020023 and by an NSF grant CNS-2154149.}}

\date{\today}

\input{preamble}

\begin{document}

%% TITLE AND ABSTRACT
\sloppy
\maketitle
\begin{abstract}
\input{abstract}

\end{abstract}

\thispagestyle{empty}
\setcounter{page}{0}

\newpage

% TABLE OF CONTENTS
\enlargethispage{1cm}
\tableofcontents
\pagenumbering{roman}
\newpage
\pagenumbering{arabic}

%% INTRO
\parskip=0.5ex

%% INTRODUCTION
\input{introduction}

%% ANALYSIS
\input{construction}
\input{cubical}
\input{left-to-middle}

\input{LPS}

%% ACKNOWLEDGEMENTS
\input{acknowledgements}

%% BIBLIOGRAPHY

\bibliography{main}

\appendix
\input{free-action}

\end{document}

%% file: preamble.tex
\usepackage{hyperref}

\usepackage{amssymb}
\usepackage{amsmath}
\usepackage{amsthm}
\usepackage{amsfonts}
\usepackage{bm}
\usepackage{enumitem}
\usepackage{color}
\usepackage{comment}

\usepackage{float}

\usepackage{todonotes}
\usepackage{asymptote}
\usepackage{mdframed}

\usepackage{hyperref}
\usepackage{enumitem}
\usepackage{framed}
\usepackage{mdframed}
\usepackage{scrextend}
\usepackage{multirow}
\usepackage{ifthen}
\usepackage{bbm}
\usepackage{frcursive}
\usepackage{thm-restate}
\usepackage{booktabs}
\usepackage{array}
\usepackage{graphicx}

\definecolor{denim}{rgb}{0.08, 0.38, 0.74}
\definecolor{classicrose}{rgb}{0.98, 0.8, 0.91}
\definecolor{darkpastelblue}{rgb}{0.47, 0.62, 0.8}
\definecolor{dogwoodrose}{rgb}{0.84, 0.09, 0.41}

\usepackage{hyperref}
\hypersetup{
    colorlinks=true,
    linkcolor=dogwoodrose,
    filecolor=dogwoodrose,
    citecolor=denim,
    urlcolor=denim,
}
\usepackage[hyperpageref]{backref}

\Crefname{theorem}{Theorem}{Theorems}
\Crefname{claim}{Claim}{Claims}
\Crefname{lemma}{Lemma}{Lemmas}
\Crefname{proposition}{Proposition}{Propositions}
\Crefname{corollary}{Corollary}{Corollaries}
\Crefname{definition}{Definition}{Definitions}

\renewcommand{\le}{\leqslant}
\renewcommand{\leq}{\leqslant}
\renewcommand{\ge}{\geqslant}
\renewcommand{\geq}{\geqslant}

\newcommand{\Signature}{\mathrm{Signature}}

\newcommand{\cC}{\mathcal{C}}

\makeatletter
\newcommand{\customlabel}[2]{%
   \protected@write \@auxout {}{\string \newlabel {#1}{{#2}{\thepage}{#2}{#1}{}} }%
   \hypertarget{#1}{#2}
}
\makeatother

\newcounter{datacounter}

\newcounter{coingame}

\newcounter{casenumb}

\newcounter{subcasenumb}

%% file: abstract.tex
We give the first construction of explicit constant-degree lossless vertex expanders.
Specifically, for any $\varepsilon > 0$ and sufficiently large $d$, we give an explicit construction of an infinite family of $d$-regular graphs where every small set $S$ of vertices has $(1-\varepsilon)d|S|$ neighbors (which implies $(1-2\varepsilon)d|S|$ unique-neighbors).
Our results also extend naturally to construct biregular bipartite graphs of any constant imbalance, where small sets on each side have strong expansion guarantees.
% These are $d$-regular graphs with the property that every small set $S \subseteq V$ has $(1-\eps) d |S|$ neighbors. \rnote{hm, i'm not super happy atm with the lead in here. it's not clear that two-sided is a strengthening}
% More generally, we construct explicit two-sided lossless expanders of constant degree and any constant imbalance.
% Specifically, for any $\eps > 0$ and large enough $d_1, d_2$, we give an explicit construction of an infinite family of $(d_1,d_2)$-biregular graphs where every small set $S$ on the left has $(1-\eps)d_1|S|$ neighbors and every small set $T$ on the right has $(1-\eps)d_2|T|$ neighbors.
The graphs we construct admit a \emph{free group action}, and hence realize new families of quantum LDPC codes of Lin and M.\ Hsieh \cite{LH22b} with a linear time decoding algorithm.

Our construction is based on taking an appropriate product of a constant-sized lossless expander with a base graph constructed from Ramanujan Cayley cubical complexes.

%% file: introduction.tex
\section{Introduction}

In this work, we give the first construction of explicit constant-degree lossless vertex expanders, thus resolving a longstanding open problem; see, e.g., \cite[Open problem 10.8]{HLW06}, and also \cite{Din24,Sri25}.
Intuitively, a graph exhibits strong vertex expansion if every sufficiently small subset of its vertices has many distinct neighbors.
Formally, a $d$-regular graph $G = (V,E)$ is called a \emph{$\gamma$-vertex expander} if there exists a small constant $\eta > 0$ (depending only on $d$) such that every subset $S \subseteq V$ of size at most $\eta |V|$ has at least $\gamma d |S|$ distinct neighbors.
We will call an infinite family of graphs \emph{lossless expanders} if $\gamma$ can be chosen as $1-\eps(d)$ for $\eps(d)\to 0$ as $d\to\infty$.
Note also that $(1-\eps)$-vertex expansion implies $(1-2\eps)$-unique-neighbor expansion.\footnote{A unique-neighbor of a set $S$ is a vertex with exactly one edge to $S$.
This property is needed in several applications, as even $\frac{1}{2}$-vertex expanders can have small subsets with zero unique-neighbors (see \Cref{sec:history}).
}

Our main result is stated as follows:
\begin{mtheorem}[Constant-degree lossless expanders]
\label{thm:main-intro}
    For every $\eps > 0$, there exists a sufficiently large integer $d_0$ such that for every integer $d\ge d_0$, there is an explicit (deterministic polynomial-time constructible) infinite family of $d$-regular graphs $G$ that are $(1-\eps)$-vertex expanders.
\end{mtheorem}
In fact, we prove a stronger statement: we construct \emph{two-sided lossless expanders of arbitrary constant imbalance}.
Concretely, a $(d_L, d_R)$-biregular bipartite graph $G = (L, R, E)$ is a two-sided lossless expander if any sufficiently small subset $S \subseteq L$ has at least $(1 - \eps)d_L|S|$ neighbors in $R$, and likewise, any sufficiently small subset $S \subseteq R$ has at least $(1 - \eps)d_R|S|$ neighbors in $L$.
More generally, for each constant $\beta\in(0,1]$ and ``many'' large enough $d_L,d_R$ for $d_R\approx \beta d_L$, we construct an infinite family of $(d_L,d_R)$-biregular two-sided lossless expanders; see \Cref{thm:main} for details.
Observe that when $d_L = d_R$, this recovers the above standard notion of lossless expansion.

Our construction also admits a \emph{free group action} by a group of size linear in the number of vertices in the graph, resolving a conjecture of Lin and M.\ Hsieh \cite[Conjecture 10]{LH22b}.
By their work, our construction yields a new family of good quantum LDPC codes, which also admit a linear time decoding algorithm; see \Cref{app:free-action} for details.

\subsection{History of vertex expanders}
\label{sec:history}
% We now put our results into context.

The quest for explicit lossless vertex expanders can be traced back to the seminal work of Sipser and Spielman~\cite{SS96} who identified vertex expansion as an important property for error correction.
In particular, they showed that a \emph{one-sided} lossless expander can be used to construct a good error-correcting code with a linear-time decoding algorithm.
Around the same time, a parallel line of work on distributed routing in networks \cite{Pip93,ALM96,BFU98} identified vertex expansion as a crucial property of networks for designing routing protocols.
At the time, it was well understood that a random graph is a lossless vertex expander with optimal parameters with high probability, but no explicit constructions were known.

\parhead{The quest for explicit constructions I.}
The first work in the direction of obtaining explicit constructions was by Kahale \cite{Kah95}, who proved that any $d$-regular Ramanujan graph is a $(1/2-o(1))$-vertex expander.
Unfortunately, this barely fell short of being useful for applications, which needed small sets to have many \emph{unique-neighbors}.
In the same work, Kahale proved that $1/2$ was an inherent barrier to spectral techniques by constructing a near-Ramanujan graph along with a small subset $S$ of vertices with only $d/2 \cdot |S|$ neighbors, and more strikingly, with \emph{zero} unique-neighbors (see \cite{MM21,KK22,KY24} for similar examples of such graphs).

The first explicit construction of unique-neighbor expanders was given by Alon and Capalbo \cite{AC02}.
Shortly after, in a breakthrough work, Capalbo, Reingold, Vadhan, and Wigderson \cite{CRVW02} gave explicit constructions of one-sided lossless expanders.

% Since then, the problem of constructing (fully) lossless expanders has remained open until the present work, despite being regarded as a cornerstone problem in the area (see, e.g., \cite[Open problem 10.8]{HLW06}, and also \cite{Din24,Sri25}).

\parhead{Applications.}
We refer the reader to \cite{CRVW02} for a detailed treatment of known applications of lossless expanders at the time in coding theory, distributed routing, fault tolerant networks, storage schemes, and proof complexity.

Ever since, the array of applications has expanded: \cite{DSW06,BV09} proved that one can use codes arising from unique-neighbor expanders to construct \emph{robustly testable codes}, and Viderman \cite{Vid13} gave a linear-time decoding algorithm for codes constructed from $2/3$-vertex expanders.
Vertex expanders have also seen applications in high-dimensional geometry:
the works of \cite{GLR10,Kar11,BGI+08,GMM22} used unique-neighbor expanders to construct \emph{$\ell_p$-spread subspaces} and matrices satisfying the \emph{$\ell_p$-isometry property}.
The work \cite{HMP06} gave a construction of a family of deterministic and uniform circuits for computing the (approximate) majority of $n$ bits assuming the construction of fully lossless expanders, not known to exist until the present work.
Motivated by randomness extractors, the works \cite{TUZ07,GUV09} gave constructions of polynomially imbalanced one-sided lossless expanders.

More recently, in the wake of advances on constructing $c^3$-locally testable codes \cite{DEL+22,PK22} and quantum LDPC codes \cite{PK22},
Lin and M.\ Hsieh gave alternate simpler constructions of both these objects: $c^3$-LTCs in \cite{LH22a} based on one-sided lossless expanders, and quantum LDPC codes in \cite{LH22b} based on two-sided lossless expanders with a free group action, whose first construction appears in the present work.

\parhead{The quest for explicit constructions II.}
The work of Lin and M.\ Hsieh \cite{LH22b} renewed interest in constructing vertex expanders, which led to a flurry of new work.
Asherov and Dinur \cite{AD23} gave a simple construction of one-sided unique-neighbor expanders, based on generalizing a construction in \cite{AC02}, which was simplified in a work of Kopparty, Ron-Zewi, and Saraf \cite{KoppartyRZS24}.
Golowich \cite{Gol23} and independently, Cohen, Roth and Ta-Shma \cite{CohenRTS23} proved that their construction instantiated with appropriate parameters in fact yields one-sided lossless expanders.

J.\ Hsieh, McKenzie, Mohanty, and Paredes \cite{HMMP24} generalized a different construction of \cite{AC02} to obtain two-sided unique-neighbor expanders of arbitrary imbalance, which additionally guarantee that sets of size $\exp(O(\sqrt{\log n}))$ expand losslessly.
The work of Chen \cite{Chen2024} built on their construction and improved the expansion guarantees for small polynomial-sized subsets of vertices.
More recently, J.\ Hsieh, Lin, Mohanty, O'Donnell, and Zhang \cite{HLMOZ25} constructed two-sided $(3/5-\eps)$-vertex expanders using construction ideas from \cite{HMMP24} with a base graph based on Ramanujan high-dimensional expanders of \cite{LSV05,LSV05b}, notably presenting the first construction of (two-sided) constant-degree graphs breaking Kahale's spectral barrier.

Using significantly different ideas, Chattopadhyay, Gurumukhani, Ringach, and Zhao \cite{CGRZ2024} studied the bipartite graphs of \cite{KalevT22}, which have polynomially large imbalance, and showed that they have two-sided lossless expansion --- the first construction of two-sided lossless expanders in the unbalanced setting.
% Using this, they also constructed non-bipartite lossless expanders with polynomial degree.
In contrast, we focus on bipartite graphs with constant degrees and constant imbalance.

\subsection{Cubical complexes}
\label{sec:cubical-overview}

Our construction of lossless expanders relies on expanding cubical complexes.
Here, we give a brief overview;
see \Cref{sec:cubical-complexes} for more definitions and properties, and \Cref{sec:cubical-construction} for an explicit construction using LPS Ramanujan graphs~\cite{LPS88}.

The theory of expanding cubical complexes was first studied by Jordan and \Livne \cite{JL00} as a high-dimensional generalization of Ramanujan graphs, where it was shown that infinite families of such complexes exist but no explicit construction was given.
Later, explicit constructions were presented in \cite{RSV19} (in a slightly different form), where more general cases were also treated.
Recently, cubical complexes were used in \cite{DLV24} to construct quantum locally testable codes, and they instantiated the complexes using abelian lifts of expanders \cite{JMOPT22}.

Earlier, a $2$-dimensional version of the cubical complexes, dubbed \emph{left-right Cayley complexes}, was an important ingredient in the constructions of locally testable codes with constant rate, distance and locality, as well as good quantum LDPC codes by \cite{DEL+22,PK22}. For our purposes, we will need higher-dimensional cubical complexes with constant degree and good expansion; notably, these can only be constructed over non-abelian groups.

\parhead{Cayley cubical complex.\footnote{One can define cubical complexes from any set $\Gamma$ and sets of permutations of $\Gamma$. For simplicity, we restrict to Cayley cubical complexes.}}
A $k$-dimensional cubical complex can be constructed from a finite group $\Gamma$ and generating sets $A_1, A_2,\dots,A_k \subseteq \Gamma$ that satisfy
\begin{enumerate}[(1)]
    \item $A_i \cdot A_j = A_j \cdot A_i$ for all $i\neq j$, and
    \item $|A_1 \cdots A_k| = |A_1| \cdots |A_k|$.
\end{enumerate}
Here, we denote $A \cdot B = \{a b: a\in A,\ b\in B\}$.
We call any collection of sets $A_1,\dots,A_k$ satisfying the above \emph{cubical generating sets}.
Note that we require $A_1,\dots,A_k$ to commute as sets while the elements do not necessarily commute.
In particular, for any $a_1\in A_1$ and $a_2 \in A_2$, there exist unique $b_1\in A_1$ and $b_2 \in A_2$ such that $a_1 a_2 = b_2 b_1$.
More generally, for any $\{a_i\in A_i\}_{i\in[k]}$ and any permutation $\pi \in S_k$, there exist unique $\{b_i\in A_i\}_{i\in [k]}$ such that $a_1 a_2 \cdots a_k = b_{\pi(1)} b_{\pi(2)} \cdots b_{\pi(k)}$.
% Note that cubical generating sets are easy to obtain for abelian groups, but it is well-known that constant-degree Cayley graphs on abelian groups cannot be expanders.

Given a group $\Gamma$ and cubical generating sets $A_1,\dots,A_k\subseteq \Gamma$, the \emph{decorated}\footnote{We use the word ``decorated'' since the vertex set $X(0)$ comprises $2^k$ copies of $\Gamma$, unlike traditional Cayley graphs that have only one copy of $\Gamma$.}
\emph{cubical complex}, denoted $X = \Cay(\Gamma; (A_1,\dots,A_k))$, is the complex with vertex set $X(0) = \Gamma \times \zo^{k}$, edges of the form $\{(g, x), (ga_i, x\oplus e_i)\}$ where $g\in \Gamma$ and $a_i\in A_i$, and $k$-faces (or cubes) $X(k)$ of the form $f = \{(f_x, x)\}_{x\in \zo^k}$ where $f_x^{-1} f_{x\oplus e_i} \in A_i$ for each $i\in[k]$ and $x\in\zo^k$.
It is easy to verify that the requirements of cubical generating sets imply that each cube is uniquely specified by a group element $g\in \Gamma$ and $\{a_i\in A_i\}_{i\in[k]}$.
See \Cref{def:cubical-complex} for a formal definition and \Cref{fig:cubical-complex} for an illustration.

\begin{figure}[ht]
    \centering
    % First subfigure
    % \begin{subfigure}{0.45\textwidth}
    %     \centering
    %     \includegraphics[width=\textwidth]{Figures/cubical.pdf}
    %     \caption{A cube in the complex.}
    % \end{subfigure}
    % \hfill
    % Second subfigure
    % \begin{subfigure}{0.3\textwidth}
    %     \centering
    %     \includegraphics[width=\textwidth]{Figures/product.pdf}
    %     \caption{The product $G \diamond H$.}
    % \end{subfigure}
    \centering
    \includegraphics[width=0.45\textwidth]{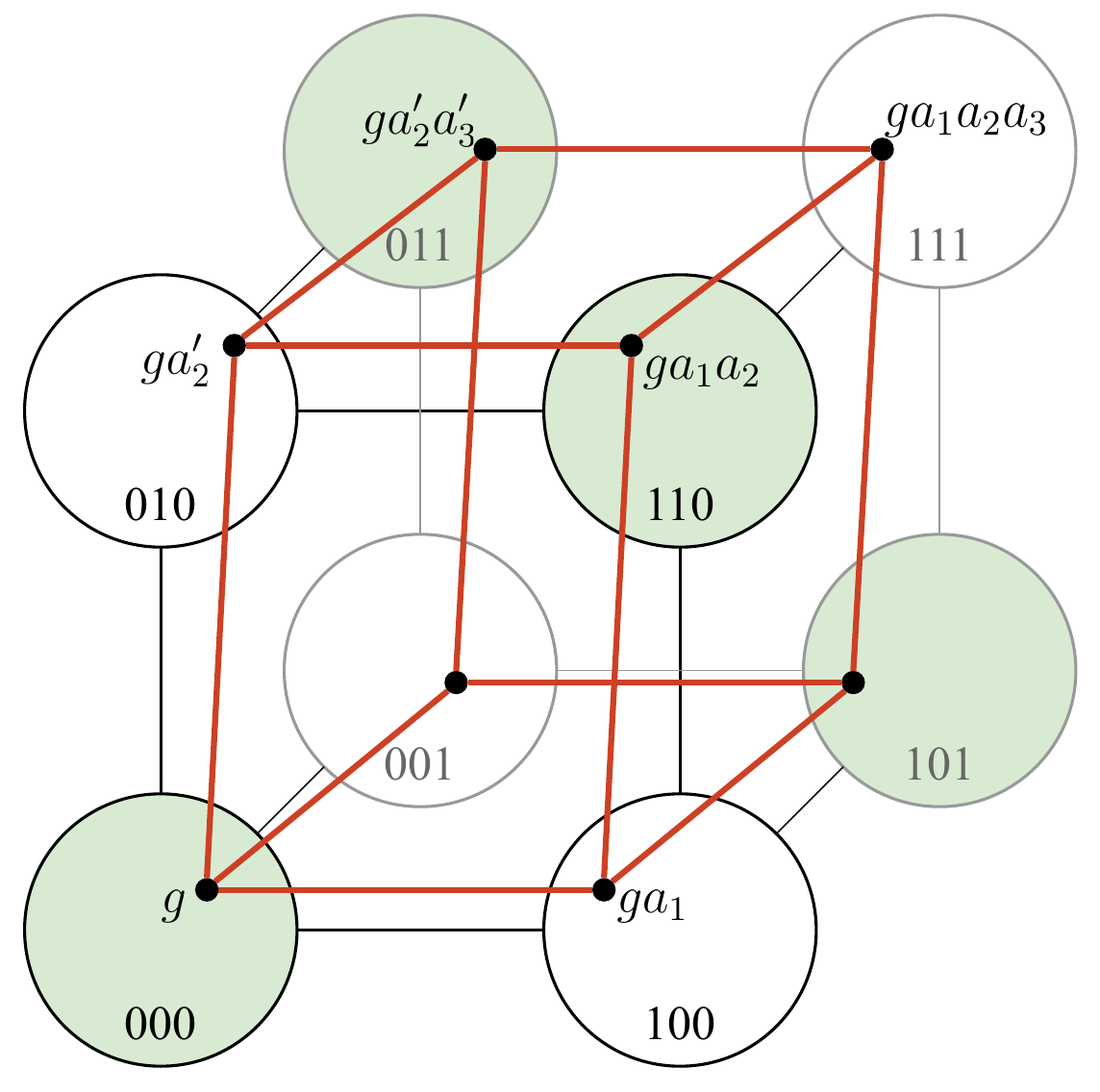}
    \caption{A $3$-dimensional (decorated) cubical complex $X = \Cay(\Gamma; (A_1,A_2,A_3))$, where the vertex set $X(0) = \Gamma \times \F_2^3$.
    An element $g\in \Gamma$ and $a_1 \in A_1$, $a_2\in A_2$, $a_3\in A_3$ uniquely specify a face (or cube) $f \in X(3)$, as depicted in the figure.
    Note that by the properties of $A_1,A_2,A_3$, there exist unique $a_1' \in A_1$, $a_2' \in A_2$ and $a_3' \in A_3$ such that $a_1a_2a_3 = a_2' a_3' a_1'$. \\
    \hspace*{1em} The vertex-face incidence graph we need for our base graph construction will be restricted to a linear code $\calC \subseteq \F_2^k$ of large distance --- the bipartite graph between $X(k)$ and $\Gamma \times \calC \subseteq X(0)$ where edges indicate containment. Here, a code $\{000, 011, 110, 101\}$ is highlighted.}
    \label{fig:cubical-complex}
\end{figure}

We note that it is straightforward to construct cubical complexes using abelian groups since all elements commute.
However, we need the complex to exhibit strong expansion, and it is well known that constant-degree abelian Cayley graphs cannot be expanders  \cite{AR94}.

We construct cubical complexes based on the LPS Ramanujan graphs \cite{LPS88}.
\Cref{sec:cubical-construction} contains an exposition and self-contained proofs of the properties we need.
Here, we briefly recall that given primes $p,q \equiv 1 \pmod{4}$, the LPS graphs $X(p;q)$ are Cayley graphs over $\Gamma = \PSL(2,\F_q)$ with $p+1$ generators $A(p)$.
The Ramanujan cubical complex we construct is simply $\Cay(\Gamma; A(p_1), A(p_2), \dots, A(p_k))$ for distinct primes $p_1,\dots,p_k$.
It is a remarkable fact that $A(p_1),\dots,A(p_k)$ indeed form cubical generating sets as defined above (\Cref{lem:LPS-cubical-generating}).
Moreover, since each Cayley graph $\Cay(\Gamma; A(p_i))$ is Ramanujan (a fact that we will only use as a black box), the resulting Ramanujan cubical complexes also inherit strong expansion properties.

\begin{remark}
    By substituting the (arguably more elementary) cubical complex from~\cite[Section 3.5.2]{DLV24}---derived from abelian lifts of $\Theta(\log n)$-sized Ramanujan Cayley graphs~\cite{JMOPT22}---into our construction, one obtains constant-degree $n$-vertex graphs in which every subset of size $O\parens*{n/\mathrm{polylog}\,n}$ has lossless vertex expansion, and which supports a free group action by a $\Theta(n/\mathrm{polylog}\,n)$-sized group.
\end{remark}

\subsection{Our construction of lossless expanders}
\label{sec:construction-overview}

Our construction is based on the \emph{tripartite line product} framework of~\cite{HMMP24}, which is a generalization of the \emph{line product} introduced in \cite{AC02}.
The first component is an (infinite family of) tripartite base graph $G$ on vertex set $L \cup M \cup R$ (representing the left, middle, and right vertex sets), where we place a $(k, D_L)$-biregular graph $G_L$ between $L$ and $M$, and a $(D_R, k)$-biregular graph $G_R$ between $M$ and $R$.
The second component is a constant-sized gadget graph $H$, which is a $(d_L, d_R)$-biregular graph on vertex set $[D_L] \cup [D_R]$.
The tripartite line product between $G$ and $H$, denoted $Z = G \diamond H$, is the $(k d_L, kd_R)$-biregular graph on $L$ and $R$ obtained as follows: for each vertex $v\in M$, place a copy of $H$ between the $D_L$ left neighbors of $v$ and the $D_R$ right neighbors of $v$ (see \Cref{def:tripartite-line} and \Cref{fig:product} for an illustration).\footnote{We require that for each vertex $v\in M$, there exists a labeling of its left neighbors in $G_L$ and right neighbors in $G_R$ that specifies how to ``place'' the copy of $H$. It is important in our construction that $H$ is \emph{not} placed arbitrarily.}

\begin{figure}[ht]
    \centering
    % First subfigure
    \begin{subfigure}{0.58\textwidth}
        \centering
        \includegraphics[width=\textwidth]{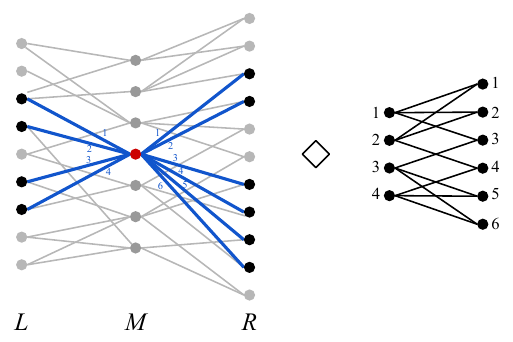}
        \caption{The base graph $G$ and gadget graph $H$.}
    \end{subfigure}
    \hfill
    % Second subfigure
    \begin{subfigure}{0.3\textwidth}
        \centering
        \includegraphics[width=\textwidth]{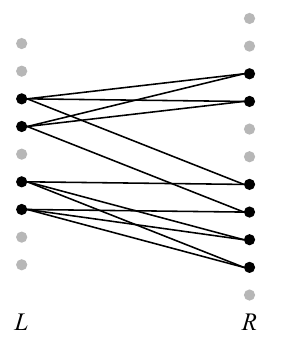}
        \caption{The product $G \diamond H$.}
    \end{subfigure}
    \caption{The tripartite line product between a base graph $G$ and gadget graph $H$.
    In this figure, only the edges from the copy of $H$ placed at the red vertex in $M$ are drawn.}
    \label{fig:product}
\end{figure}

Since the gadget graph $H$ is of constant size, we can find an $H$ that satisfies strong expansion properties by brute force.
Since a random biregular graph satisfies our desired properties with high probability, it is convenient to think of $H$ as a random graph.
The bipartite graphs $G_L$ and $G_R$ of the base graph are chosen to be explicit bipartite expanders.
In \cite{HMMP24}, they are chosen to be explicit near-Ramanujan bipartite graphs~\cite{LPS88,Mor94}, while in \cite{HLMOZ25}, they are chosen to be the vertex-face incidence graphs of the $4$D Ramanujan complex from \cite{LSV05,LSV05b}.

In our case, we choose $G_L$, $G_R$ to be ``coded'' vertex-face incidence graphs of expanding cubical complexes described in \Cref{sec:cubical-overview}.
% , which are generalizations of the square complexes used in \cite{DEL+22,PK22}.
% We first give an overview of these complexes.

\parhead{Coded incidence graphs.}
We construct the bipartite base graphs $G_L, G_R$ using a $k$-dimensional Ramanujan cubical complex $X$ and the Hadamard code $\cC \subseteq \zo^k$ (with $|\cC| = k = 2^r$ for some $r\in \N$).
We set $L = X(k)$, the $k$-faces of $X$, and $M = \Gamma \times \cC$, a subset of vertices $X(0)$ according to the code $\cC$.
A $k$-face $f \in L$ and a vertex $(g,x) \in M$ are connected in $G_L$ if and only if $(g,x)\in f$.
Thus, each $f\in L$ has degree $|\cC| = k$, and each vertex in $M$ has degree $D_L = \prod_{i=1}^k |A_i|$.
The other bipartite graph $G_R$ is defined the same way.

\begin{remark}
    Restricting the vertices according to the Hadamard code $\cC$ provides crucial symmetry in our construction.
    In particular, for two vertices $(g, x)$ and $(h, y)$ with $x \neq y\in \cC$, their common neighborhood (i.e., the set of $k$-faces containing them) is either empty or all possible completions to a full cube.
    Since $\mathrm{dist}(x,y) = k/2$ for all $x\neq y\in \cC$,\footnote{We expect that any $\delta$-balanced linear code with a small enough constant $\delta$ will work as well; see \Cref{rem:balanced-code}.}
    the common neighborhoods are all roughly the same structure (by choosing $|A_1|,\dots,|A_k|$ to be a constant factor away from each other).
    We believe that this is one key improvement over \cite{HLMOZ25} which is based on Ramanujan simplicial complexes, where $M$ is also $k$-partite but the common neighborhoods (a.k.a.\ links) of two vertices differ drastically depending on which parts they are in.
\end{remark}

\subsection{Overview of the analysis}
\label{sec:analysis-overview}

Our analysis follows the same outline as \cite{HMMP24,HLMOZ25}.
To bound the expansion of a set $S \subseteq L$ (sets on the right follow the same analysis), we split into two parts: the \emph{left-to-middle} and the \emph{middle-to-right} analysis.
% Suppose the gadget graph $H$ has lossless expansion for sets of size up to $t_H$.
Fix a (small) subset $S \subseteq L$, and consider the neighbors $U = N_{G_L}(S) \subseteq M$.
For each $u\in U$, as long as $\deg_S(u) \coloneqq |S \cap N_{G_L}(u)|$ is sufficiently small, we will have lossless expansion within the gadget placed on $u$ (since the gadget is random-like). On the other hand, if $\deg_S(u)$ is too large, then the gadget cannot experience lossless expansion because the number of right vertices in the gadget is much smaller than the number of edges in the gadget arising from $N_{G_L}(u)$. 
Thus, we split $U$ into $U_{\ell}$ (low-degree) and $U_h$ (high-degree), and we need to show that most elements of $S$ partake in many $U_\ell$ gadgets and few $U_h$ gadgets: precisely, we need to show that $e_{G_L}(S, U_h)$ is small such that $1-\eps$ fraction of edges from $S$ go to $U_\ell$. 

\parhead{Left-to-middle analysis: small-set subcube density.}
We bound the \emph{small-set subcube density} of the cubical complex, similar to the triangle density bound of the Ramanujan simplicial complexes needed in \cite{HLMOZ25}.
Our goal is to show that there are not too many $k$-faces that have many vertices in $U_h$.
More specifically, we upper bound the size of $\{f\in X(k): |f \cap U_h| \geq 2\sqrt{k}\}$ by $O_k(1) \cdot D_L^{5/8} |U_h|$.
This is proved in \Cref{sec:subcube-density} using the structure and expansion of $X$.
More specifically, whereas~\cite{HLMOZ25} used spectral properties within the links of the high dimensional expander to obtain their bounds, our complex notably is not a high dimensional expander as the links are disconnected. Instead, we rely on the Hadamard structure of the links along with a variant of the Loomis--Whitney inequality~\cite{LW49} to argue that $U_h$ contains few subcubes.

To demonstrate the key ideas, we focus on the simple case of $k=3$ ---- subcube density of $3$-dimensional \emph{expanding} cubical complexes with a code $\calC = \{000, 011, 110, 101\} \subseteq \F_2^3$, as depicted in \Cref{fig:cubical-complex}.
For any small subset $U \subseteq \Gamma \times \calC$, we will show an upper bound on the size of $\{f\in X(3): |f \cap U| = 4\}$.
For simplicity, assume that $|A_1| = |A_2| = |A_3| = p$ (this is true in our construction up to absolute constants), and denote $U_x \coloneqq U \cap (\Gamma \times \{x\})$ for $x\in \calC$.

First, we use the expansion property of the cubical complex.
Consider the bipartite graph between $\Gamma \times \{000\}$ and $\Gamma \times \{110\}$, where $(g, 000)$ and $(ga_1a_2, 110)$ are connected for $a_1\in A_1$ and $a_2\in A_2$.
This bipartite graph has degree $|A_1| \cdot |A_2| = p^2$ and has second eigenvalue $O(p)$, which implies that the subgraph induced by $U_{000} \cup U_{110}$ has average degree $O(p)$.
Thus, a typical element $(g, 000) \in U_{000}$ has at most $O(p)$ neighbors in $U_{110}$, $U_{101}$ and $U_{011}$ respectively.
% Thus, the key question is: for a collection of cubes $T \subseteq X(3)$ that contain $(g, 000)$, suppose $|\{f_x: (f_x,x) \in f \text{ for some $f\in T$} \}|$

The next crucial property we use is the fact that any cube $f$ is uniquely identified by any $3$ points in $f \cap (\Gamma \times \calC)$.
For example, $(g, 000)$, $(ga_1 a_2, 110)$ and $(g a_1 a_3, 101)$ uniquely specifies a cube $f \in X(3)$, and in particular, there exist unique $a_2'\in A_2$ and $a_3' \in A_3$ such that $(g a_2' a_3', 011) \in f$.
For simplicity, let us assume that $a_2' = a_2$ and $a_3' = a_3$.
Then, the key question is:
\begin{displayquote}
    For a set of $3$-tuples $T$, suppose $N_{12} = |\{(a_1, a_2): (a_1, a_2, a_3) \in T \text{ for some $a_3$}\}|$ and $N_{13}, N_{23}$ defined similarly, how large can $T$ be?
\end{displayquote}
The answer is $|T| \leq \sqrt{N_{12} N_{23} N_{13}}$.
This is in fact a special case of the \emph{Loomis--Whitney inequality}.
Here, we give a simple proof using an entropic argument.
For the uniform distribution over $T$, we have $H(a_1,a_2,a_3) = \log |T|$, while by assumption $H(a_i, a_j) \leq \log N_{ij}$ for $i < j$.
The well-known Shearer's inequality states that $H(a_1,a_2,a_3) \leq \frac{1}{2} \sum_{i<j} H(a_i, a_j)$, which completes the proof.

Our argument for general $k$ follows the same idea.
The reason that $2\sqrt{k}$ is relevant is because for any subset $B \subseteq \cC$ of a \emph{linear} code $\cC \subseteq \F_2^k$ with $|B| \geq 2\sqrt{|\cC|}$, there exist four distinct elements $\sigma_1,\sigma_2,\sigma_3,\sigma_4\in B$ such that $\sigma_1 + \sigma_2 + \sigma_3 + \sigma_4 = 0$ (\Cref{lem:s1-s2-s3-s4}).
This, at a high level, reduces to the $3$-dimensional case.
We are able to show that $|\{f\in X(k): |f\cap U| \geq 2\sqrt{k}\}| \leq O_k(1) \cdot D_L^{5/8} |U|$.
Thus, by setting the threshold for $U_{\ell}$ and $U_h$ to be larger than $D_L^{5/8}$ and $k = O(1/\eps^2)$, we have that most vertices in $S \subseteq L$ have at least $1 - \frac{2\sqrt{k}}{k} \geq 1-\eps$ fraction of edges going to $U_\ell$.
This completes the left-to-middle analysis.

\parhead{Middle-to-right analysis.}
Having established that most vertices of $S$ participate in many low-degree gadgets, it remains to show that these different gadgets do not have too many collisions in $G_R$.
Our proof of this part closely follows the \emph{middle-to-right analysis} in \cite{HLMOZ25}.
In fact, as noted in \cite{HLMOZ25}, the common neighborhood structure of $G_R$ is the key improvement over \cite{HMMP24} which uses Ramanujan bipartite graphs.

It is convenient to view the expansion of each gadget $H_u$, for $u\in U$, as ``red'' edges going from $u$ to vertices in $N_{G_R}(u) \subseteq R$.
The neighbors of $S$ in the final product $Z$ are exactly the vertices incident to any red edge.
See \Cref{fig:collision-graph} for an example.
The red edges form a subgraph of $G_R$, denoted $\RED$, and we need to show that there are very few collisions on the right.

To this end, we define a \emph{collision} (multi-)graph $C$ on $U$, where we place an edge $\{u,v\}$ for each $u \neq v\in U$ and $r\in R$ such that $\{u,r\}, \{v,r\} \in \RED$ (see e.g.\ \Cref{fig:collision}).
We need to show an upper bound on $e(C)$.
Let $\ul{C}$ be the simple graph obtained by removing duplicated edges from $C$.
Moreover, let $\wt{G}_R$ be the simple graph on $M$ where $u\neq v\in M$ are connected if they have a common neighbor in $R$.
Observe that $\ul{C}$ is a subgraph of $\wt{G}_R$.
Then, the natural idea to bound $e(C)$ is to use the expansion of $\wt{G}_R$, which we call \emph{skeleton expansion} (\Cref{def:skeleton-expansion}).

If $G_R$ is chosen to be a Ramanujan bipartite graph (as in \cite{HMMP24}), then most pairs of vertices in $M$ have few common neighbors, and $\wt{G}_R$ has degree $O(D)$ and second eigenvalue $O(\sqrt{D})$.
In our case, due to the structure of the cubical complexes, every pair of vertices in $M$ has either zero or $\approx \sqrt{D}$ common neighbors, and thus $\wt{G}_R$ has degree $O(\sqrt{D})$ and second eigenvalue $O(D^{1/4})$.
This is the key improvement over \cite{HMMP24}.
Of course, now the collision graph $C$ may have large multiplicities, which complicate the analysis.
We handle this by using the spreadness of the ``random'' gadget $H$ (\Cref{lem:pr-gadget}), and crucially this requires us to place the gadget in the same way for every $u\in M$ (as opposed to arbitrarily).
See \Cref{sec:main-proof} for more details.

\subsection{Discussion and future directions}

In this work, we constructed graphs with good vertex expansion, namely, that every small set of vertices has many neighbors. Notably, by using the high dimensional structure of cubical complexes, we were able to bypass the spectral limitations of considering only the $1$-dimensional structure. A related problem we find fascinating is whether we can construct \emph{edge expanders} beyond what is guaranteed by spectral techniques.

% In this work, we constructed lossless expanders by appealing to the higher dimensional structure of expander-like objects. In some sense, this allowed us to bypass spectral limitations of considering only the $1$-skeleton. Nonetheless, we find the question of constructing ($1$-dimensional) graphs that closely mimic the properties of random graphs beyond what its spectral properties guarantee a fascinating question. We state one such question below.

% An insight from the present work, as well as recent advances in quantum codes \cite{PK22,DLV24}, locally testable codes \cite{DEL+22}, PCPs \cite{BMV24}, and vertex expanders \cite{HLMOZ25}, is that high-dimensional expander-like objects can be an effective amplifier to lift a constant-sized object satisfying certain desirable properties into a large object with the same properties.
% In light of this message, we propose a few open directions for future work.

% \parhead{Euclidean sections.}
% Every vector in a random subspace of dimension $\Omega(n)$ in $\R^n$ is ``delocalized'', but current explicit constructions for these objects are far from optimal; see, e.g., \cite{GLR10,GMM22} and the references therein for a history of this problem, and its applications to error correction over reals, compressed sensing, and nearest neighbors algorithms in high dimensions. \rnote{i think we either need to justify why this problem is related to lossless expansion, or we should omit it.}

\parhead{Ultra-lossless edge expanders.}
In a random $d$-regular graph, any sufficiently small set $S$ has at least $(d-1-\eps)|S|$ edges leaving $S$.
In contrast, small sets $S$ in Ramanujan graphs have $(d-O(\sqrt{d}))|S|$ edges leaving $S$.

We call an expander satisfying the benchmark set by random graphs an \emph{ultra-lossless edge expander}.
One can prove that an ultra-lossless edge expander is also a lossless vertex expander.
While it is unclear if they unlock more applications, we believe explicit constructions of them would likely introduce novel ideas. 

\parhead{High dimensional amplification for further applications?}
An insight from this work, as well as recent advances in quantum codes \cite{PK22,DLV24}, locally testable codes \cite{DEL+22,PK22,LH22a}, PCPs \cite{BMV24}, and vertex expanders \cite{HLMOZ25}, is that high-dimensional expander-like objects can be an effective amplifier to lift a constant-sized object satisfying certain desirable properties into a large object with the same properties. This local-to-global lifting has long been known for (1-dimensional) expanders in many contexts (e.g.\ \cite{SS96,AEL95,GLR10,GMM22}), though for other applications 1-dimensional expansion have not proved sufficient. We hope that the ideas from the present work on the usage of high dimensional structures as a local-to-global amplifier will unlock new applications across theoretical computer science and mathematics.

%% file: construction.tex
\section{Construction of lossless vertex expanders} \label{sec:construction}

Our main result is the construction of explicit two-sided lossless expanders.
We first formally define two-sided vertex expanders.

\begin{definition}\label{def:two-sided}
    A family of $(d_L, d_R)$-biregular bipartite graphs $Z = \{ Z_n = (L_n, R_n, E_n) \}$ is a \emph{two-sided $\gamma$-vertex expander} if there is some $\eta > 0$ depending only on $d_L, d_R, \gamma$ for which the following holds:
    \begin{itemize}
        \item For any $S \subseteq L$ of size $|S| \le \eta \cdot |L|$, $S$ has $\ge \gamma d_L |S|$ neighbors on the right,
        \item For any $T \subseteq R$ of size $|T|  \le \eta \cdot |R|$, $T$ has $\ge \gamma d_R |T|$ neighbors on the left.
    \end{itemize}
    When we can take $\gamma = 1-\eps(d)$ for $\eps(d) \to 0$ as $d\to\infty$, we refer to $Z$ as a \emph{two-sided lossless expander}.
\end{definition}

Our main result is stated below.

\begin{theorem} \label{thm:main}
    For every $\eps, \beta\in (0, 1]$, there exists $k=k(\eps), d_0 = d_0(\eps,\beta) \in \bbN$ such that for any $d_L, d_R \ge d_0$ for which $\beta \le d_L/d_R \le \beta+\eps$, there is an infinite family of graphs $(kd_L, kd_R)$-biregular bipartite graphs $(Z_n)_{n\ge 1}$ for which $Z_n$ is a two-sided $(1-\eps)$-vertex expander on $\Theta(n)$ vertices.
    Additionally, there is an algorithm that takes in a positive integer $n$ as input, and in $\poly(n)$-time outputs $Z_n$.
\end{theorem}

\begin{remark}
    In the special case where $d_L = d_R = d$, the construction can be made $d$-regular for \emph{any} $d \ge d_0(\eps)$ (as stated in \Cref{thm:main-intro}).
    The trick is to begin with a $\wt{d}$-bipartite graph $G$ guaranteed by \Cref{thm:main} where $\wt{d} \in \bracks*{d, \parens*{1+\frac{1}{k-1}}d }$.
    Since $G$ is bipartite, it can be decomposed into $\wt{d}$ edge-disjoint perfect matchings.
    By taking the union of any $d$ of these matchings, we obtain a $d$-regular subgraph.
    Such a $d$-regular subgraph can be seen to incur only a negligible loss in expansion.
\end{remark}

As mentioned in the introduction, our construction also admits a \emph{free group action} by a group of size linear in the number of vertices in the graph.
By the work of \cite{LH22b}, our construction yields a new family of good quantum LDPC codes that admit linear-time decoding algorithms; see \Cref{app:free-action} for details.

Our construction of lossless expanders is based on the \emph{tripartite line product}, introduced in \cite{HMMP24}.
See \Cref{fig:product} for an example.

\begin{definition}[Tripartite line product]
\label{def:tripartite-line}
    Given the ingredients:
    \begin{itemize}
        \item two bipartite \emph{base graphs}, a $(k,D_L)$-biregular graph $G_L = (L, M, E_L)$, and a $(k,D_R)$-biregular graph $G_R = (R, M, E_R)$, along with injective functions $\mathrm{LNbr}_u:[D_L]\to L$ and $\mathrm{RNbr}_u:[D_R]\to R$ for every vertex $u\in M$ that index the left and right neighbors of $u$, 
        \item a $(d_L,d_R)$-biregular \emph{gadget graph} $H$ where the left-hand side is $[D_L]$, and the right-hand side is $[D_R]$,
    \end{itemize}
    we define the \emph{tripartite line product} of $(G_L,G_R)$ and $H$ as the $(kd_1,kd_2)$-biregular graph $Z$ obtained by taking each middle vertex $u\in M$, and placing a copy of $H$ between the left and right neighbors of $u$.
    Specifically, for every edge $(i,j)\in H$, we place an edge between $\mathrm{LNbr}_u(i)$ and $\mathrm{RNbr}_u(j)$.
\end{definition}

Our construction is obtained as the tripartite product of bipartite graphs arising from \emph{Ramanujan cubical complexes} with a constant-sized gadget graph, which can be thought of as a random graph.

\subsection{Base and gadget graph constructions}

In this section, we describe the precise properties we will need from the bipartite graphs and the gadget graph.

\parhead{Notation, terminology, and parameters.}
Given a graph $G$ and $S,T\subseteq V(G)$, we use $G[S]$ to refer to the induced subgraph of $G$ on $S$, and $G[S,T]$ as the induced bipartite subgraph of $G$ between $S$ and $T$.
Given a bipartite graph $(U,V,E)$, we denote an edge between a vertex $u\in U$ and $v\in V$ by the ordered tuple $(u,v)$.

In our construction, the parameters $k, D_L, D_R, d_L, d_R$ are all constants (large enough depending on $\eps,\beta$) compared to the size of the base graphs.
However, it is convenient to treat $k \approx \eps^{-2}$ as fixed while $d_L, d_R$ and $D \coloneqq D_L + D_R$ grow (as we want constructions for infinitely many degrees), and we will use $o_D(1)$ to denote a quantity that can be made smaller than any constant by making $D$ a large enough constant.

\parhead{Base graph construction.}
Following \cite{HLMOZ25}, we introduce the notion of a \emph{structured bipartite graph}.
\begin{definition}[Structured bipartite graph]
\label{def:structured-bipartite-graph}
    A $(k,D)$-biregular bipartite graph $G$ between vertex sets $V$ and $M$ is a \emph{structured bipartite graph} if:
    \begin{enumerate}[(1)]
        \item For each vertex $u\in M$, there is an injective function $\mathrm{Nbr}_u:[D]\to V$ that specifies an ordering of the $D$ neighbors of $u$.
        \item The set $M$ can be expressed as a disjoint union $\sqcup_{a\in[k]}M_a$ such that each $v\in V$ has exactly one neighbor in each $M_a$.
        \item \label{property:special-sets}
        There is an $s\in \N$ such that the following holds:
        for each pair of distinct $a, b\in [k]$, there are $r(a,b)$ \emph{special sets} $\{Q_i^{a,b} \subseteq [D]\}_{i\in[r(a,b)]}$ that partition $[D]$ (abbreviated to $r$ and $Q_i$), each $|Q_i| \in [\frac{D}{2s}, \frac{2D}{s}]$,
        such that for every $u \in M_a$, there are distinct $v_1,\dots,v_r \in M_b$ with $N(u)\cap N(v_i) = \Nbr_u(Q_i)$ for each $i\in[r]$ and $N(u) \cap N(v') = \varnothing$ for all other $v'\in M$.
    \end{enumerate}
\end{definition}

Intuitively, \Cref{property:special-sets} of \Cref{def:structured-bipartite-graph} means that for every $u\in M_a$, there are $r(a,b)$ vertices in $M_b$ that have common neighbors with $u$, and the common neighborhoods form a specific structure.
See \Cref{fig:special-sets} for an illustration.
For our construction, it is important that this structure is the same across all $u\in M_a$  --- the special sets $\{Q_i \subseteq [D]\}$ are independent of $u$ (but can depend on $a,b\in [k]$).

Henceforth, we fix $G$ as a structured $(k,D)$-biregular graph between $V$ and $M$.
\begin{definition}[Small-set $j$-neighbor expansion]
    We say $G$ is a \emph{$\tau$-small-set $j$-neighbor expander} if for some small constant $\eta > 0$, and for every $U\subseteq M$ such that $|U|\le \eta |M|$, the number of vertices in $V$ with at least $j$ neighbors in $U$ is bounded by $\tau\cdot|U|$.
\end{definition}

\begin{definition}[Small-set skeleton expansion]
\label{def:skeleton-expansion}
    Let $\wt{G}$ be the simple graph on $M$ obtained by placing an edge between $u,u'\in M$ if there exists a length-$2$ path between $u$ and $u'$.
    We say $G$ is a \emph{$\lambda$-small-set skeleton expander} if for some small constant $\eta > 0$, and for every $U\subseteq M$ such that $|U|\le \eta |M|$, the largest eigenvalue of the adjacency matrix of the graph $\wt{G}[U]$ is at most $\lambda$.
\end{definition}

We now state the guarantees we can achieve in a structured bipartite graph, which we prove in \Cref{sec:subcube-density}.
\begin{lemma}   \label{lem:base-graph}
    For every $k$ that is a power of $2$, and large enough $D\in\bbN$, there is an algorithm that takes in $n,D_L,D_R\in\bbN$ as input where $D_L,D_R\le D$, and constructs vertex sets $L, M, R$ such that $|M| = \Theta(n)$ and $|R| = |L|\cdot D_L/D_R$ along with structured bipartite graphs $G_L$ on $(L,M)$, $G_R$ on $(R,M)$, where $G_L$ is $(k,D_L)$-biregular and $G_R$ is $(k,D_R)$-biregular, with the following properties:
    \begin{itemize}
        \item $s = \Theta(\sqrt{D})$ for the special set structure.
        \item $G_L$ and $G_R$ are $O\parens*{D^{5/8}}$-small-set $2\sqrt{k}$-neighbor expanders.
        \item $G_L$ and $G_R$ are $O\parens*{D^{1/4}}$-small-set skeleton expanders.
    \end{itemize}
\end{lemma}

\parhead{Gadget graph construction.}
The reader should think of the gadget graph as a random graph.
Its properties were analyzed in \cite{HMMP24,HLMOZ25}, which we articulate in the following statement.
\begin{lemma}[{\cite[Lemma 2.10]{HLMOZ25}}]   \label{lem:pr-gadget}
    Let $D_L,D_R,d_L,d_R,k,s$ be integers such that $D_L\cdot d_L = D_R \cdot d_R$, and $k \leq D^{0.1} \leq d_L,d_R \leq o_D(D)$ where $D \coloneqq D_L + D_R$.
    Suppose for any distinct $a,b\in [k]$,
    there is an $r(a,b)\in \N$ and a partition $(Q_i^{a,b})_{i\in[r(a,b)]}$ of $[D_R]$ where each partition has size within $\bracks*{\frac{D}{2s}, \frac{2D}{s}}$.
    Then, there exists a bipartite graph $H$ on $[D_L]\cup[D_R]$ such that
    \begin{itemize}
        \item \textbf{(lossless expansion)} for any $A\subseteq[D_L]$ with $|A| \leq o_D(1)\cdot D_R/d_L$, we have $|N(A)| \ge (1-o_D(1))d_L|A|$,
        \item \textbf{(spread)} for any distinct $a,b\in [k]$,
        for any $A\subseteq[D_L]$ and any $W\subseteq [r(a,b)]$ with $|W| \ge \frac{s\log D}{d_L}$,
        \[
            \sum_{i\in W} |N(A)\cap Q_i| \le 32 |W| \cdot \max\braces*{ \frac{d_L|A|}{s},\ \log D }\mper
        \]
    \end{itemize}
    Additionally, $H$ satisfies the above guarantees when the roles of ``$L$'' and ``$R$'' are swapped.
\end{lemma}

The spread condition above can be interpreted as follows:
for any $A \subseteq [D_L]$ not too small, it has at most $d_L |A|$ neighbors, and any $|W|$ special sets contain at most an $O\parens*{\frac{|W|}{s}}$ fraction of them.

\subsection{Proof of \texorpdfstring{\Cref{thm:main}}{Theorem~\ref{thm:main}}}
\label{sec:main-proof}

We are now ready to use the above ingredients to prove \Cref{thm:main} on the explicit construction of $2$-sided lossless vertex expanders.
Given $\eps$, $d_L$ and $d_R$, we choose parameters $D, D_L,D_R,k \in \N$ and $\delta \in (0,1)$ such that the following relations hold.
\begin{itemize}
    \item $D_L \cdot d_L = D_R\cdot d_R$.
    \item $D = D_L + D_R$.
    \item $k \ge 16/\eps^2$ and is a power of $2$.
    \item $\displaystyle D^{-1/16}\le \delta \le o_D(1) \cdot \frac{1}{k^2}$.
    \item $\displaystyle\frac{D^{1/4}\log^2 D}{\delta} \le d_L,d_R \le \frac{\delta D^{3/8}}{\log D}\mper$
\end{itemize}
Here, we assume $d_L, d_R \geq d_0(\eps,\beta)$ for a large enough $d_0(\eps,\beta)$ such that any $o_D(1)$ term is sufficiently small.

Let $G_L = (L,M,E_L)$ and $G_R = (R,M,E_R)$ be the structured bipartite graphs constructed from the algorithm in \Cref{lem:base-graph} with parameters $k, D, n, D_L, D_R$.
Recall that $G_L$ and $G_R$ are structured bipartite graphs with $s = \Theta(\sqrt{D})$ for the special set structure and are $O\parens*{D^{5/8}}$-small-set $2\sqrt{k}$-neighbor expanders, and $O\parens*{D^{1/4}}$-small-set skeleton expanders.
In this proof, we will use $\tau = O(D^{5/8})$ to denote the small-set $2\sqrt{k}$-neighbor expansion, and $\lambda = O(D^{1/4})$ to denote the small-set skeleton expansion.

Let $H$ be a $(d_L,d_R)$-biregular bipartite graph on $[D_L]\cup[D_R]$ whose special subsets of $[D_R]$ are identical to the special subsets associated to $G_R$, and whose special subsets of $[D_L]$ are identical to the special subsets associated to $G_L$.

Looking ahead, we will need that
\begin{itemize}
    \item $\tau \leq o_D(\delta) \cdot \frac{D_R}{d_L}$ and similarly $\tau \leq o_D(\delta) \cdot \frac{D_L}{d_R}$.
    \item $\lambda \leq s\delta$,
    \item $d_L, d_R \geq \frac{1}{\delta} \max\{\lambda, \sqrt{s}\} \log D$.
\end{itemize}
One can verify that with parameters $\tau = O(D^{5/8})$, $\lambda = O(D^{1/4})$ and $s = \Theta(\sqrt{D})$ from \Cref{lem:base-graph}, our choice for $\delta$ and $D_L, D_R$ listed above satisfy all requirements.

We output the tripartite line product $Z = (L,R,E_Z)$ of $(G_L,G_R)$ with $H$.
We will establish vertex expansion of small subsets of $L$; the analysis of the vertex expansion of small subsets of $R$ is similar.

\parhead{Left-to-middle analysis.}
Let $S\subseteq L$ such that $|S| \le \eta|L|$.
Let $U\subseteq M$ be the neighbors of $S$ in $G_L$.
% , and define $\Gamma_S\coloneqq G_L[S,U]$.
We split $U$ into its ``high-degree'' part $U_h \coloneqq \braces*{v \in U :\deg_{\GLSU}(v)\ge \frac{\tau}{\delta}}$, and ``low-degree'' part $U_{\ell}\coloneqq U\setminus U_h$.

Our first step is to prove that most edges from $S$ to $U$ point to $U_{\ell}$.
\begin{claim}   \label{claim:left-to-middle}
    The number of edges in $\GLSU$ incident to $U_{\ell}$ is at least $\parens*{1-\sqrt{\delta} - 2k^{-1/2}}\cdot k |S|$.
\end{claim}
\begin{proof}
    By definition, the number of edges incident to $U_h$ in $\GLSU$ is at least $\frac{\tau}{\delta}|U_h|$.
    On the other hand, denoting $S_{\ge 2\sqrt{k}}$ to be the set of vertices in $S$ with at least $2\sqrt{k}$ neighbors in $U_h$,
    by small-set $2\sqrt{k}$-neighbor expansion of $G_L$, we have $|S_{\geq 2\sqrt{k}}| \leq \tau |U_h|$.
    Consequently, the number of edges from $S_{\ge2\sqrt{k}}$ into $U_h$ satisfies:
    \begin{align*}
        e\parens*{S_{\ge 2\sqrt{k}}, U_h} 
        \le k \abs*{S_{\geq2\sqrt{k}}}
        \le k \tau |U_h| 
        = k \delta \cdot \frac{\tau}{\delta}|U_h|
        \le k\delta \cdot e(S,U_h)
        % \le \sqrt{\delta}\cdot\frac{\tau}{\delta}|U_h| \le \sqrt{\delta} \cdot e(S, U_h) 
        \leq \sqrt{\delta} \cdot k|S| \mper
    \end{align*}
    Here, we use $k \leq 1/\sqrt{\delta}$.
    Thus, we have:
    \begin{align*}
        e\parens*{S, U_{\ell}} &= e\parens*{S, U} - e\parens*{S, U_h} \\
        &= k|S| - e\parens*{S_{\ge 2\sqrt{k}}, U_h} - e\parens*{ S_{< 2\sqrt{k}}, U_h } \\
        % &\ge k|S| - \sqrt{\delta}\cdot e(S, U_h) - 2\sqrt{k} |S| \\
        &\ge k|S| - \sqrt{\delta}\cdot k|S| - 2\sqrt{k}|S| \\
        &= \parens*{1 - \sqrt{\delta} - \frac{2}{\sqrt{k}}} \cdot k |S|\mper
        \qedhere
    \end{align*}
\end{proof}

\input{middle-to-right}
\qed

%% file: middle-to-right.tex
\parhead{Middle-to-right analysis.}
We have proved that most edges from $S$ to $U$ touch low-degree vertices, which the reader should think of as gadgets through which the expansion into $R$ is lossless.
We make this formal below.
\begin{definition}
    For $S \subseteq L$ and $U = N_{G_L}(S) \subseteq M$, if a vertex $v \in R$ is a neighbor of $S$ in the final product due to connections from the gadget $H_u$ for $u\in U$, then we color the edge $(u, v)$ red.
    The red edges form a subgraph of $G_R$, which we denote as $\RED(S)$ or simply $\RED$ when $S$ is clear from context.
    \Cref{fig:special-sets} contains an example of the subgraph $\RED$.\footnote{We note that in \cite{HLMOZ25}, they need to define ``blue'' and ``red'' edges to prove \emph{unique-neighbor} expansion.
In our case, since we will show lossless expansion, we do not need to make this distinction.}
\end{definition}
By the choice of the threshold, we have $\frac{\tau}{\delta} \leq o_D(1) \cdot D_R / d_L$, and hence, by \Cref{lem:pr-gadget}, each vertex in $U_{\ell}$ expands by at least a $(1-o_D(1))d_L$ factor.
In particular, we have,
\begin{align*}
    e(\RED) \geq \sum_{u\in U_{\ell}} (1-o_D(1)) d_L \cdot \deg_{S}(u)
    = (1-o_D(1)) d_L \cdot e_{G_L}(S, U_{\ell}) \mper
    \numberthis \label{eq:red-edges}
\end{align*}

\begin{figure}[ht]
    \centering
    % First subfigure
    \begin{subfigure}{0.5\textwidth}
        \centering
        \includegraphics[width=\textwidth]{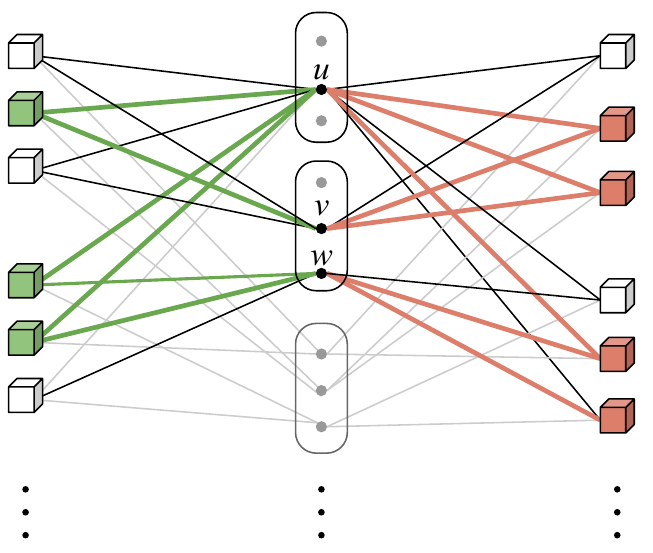}
        \caption{Let $S\subseteq L$ consist of the cubes colored green, and the cubes on the right incident to red edges are the neighbors of $S$ in the final product $Z$.}
        \label{fig:special-sets}
    \end{subfigure}
    \qquad
    % Second subfigure
    \begin{subfigure}{0.4\textwidth}
        \centering
        \includegraphics[width=\textwidth]{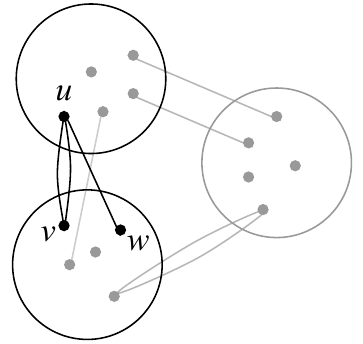}
        \caption{The collision multi-graph $C$ on $M$.
        Removing parallel edges gives the simple graph $\ul{C}$, which is a subgraph of $\wt{G}_R$.}
        \label{fig:collision}
    \end{subfigure}
    \caption{The two bipartite base graphs $G_L, G_R$ have the structure that $M$ has $k$ parts, and for $u\in M$ and $v, w \in M$ from a different part, the common neighborhoods $N_{G_R}(u) \cap N_{G_R}(v)$ and $N_{G_R}(u) \cap N_{G_R}(w) \subseteq R$ are disjoint, each corresponding to a \emph{special set} in $[D_R]$, i.e., $N_{G_R}(u) \cap N_{G_R}(v) = \Nbr_u(Q_i)$ for some special set $Q_i \subseteq [D_R]$.
    \\
    \hspace*{1em} \Cref{fig:special-sets} shows an example of $\RED(S)$, a subgraph of $G_R$.
    The middle-to-right analysis involves upper bounding the collisions of the red edges on the right.
    Here, $u$ has collisions with $v$ and $w$, represented as edges in the collision graph $C$ in \Cref{fig:collision}.
    We will show that this cannot happen too often by upper bounding $e(C)$.}
    \label{fig:collision-graph}
\end{figure}

% \begin{figure}[ht!]
%     \centering
%     \includegraphics[width=0.4\textwidth]{Figures/red.pdf}
%     \caption{Let $S\subseteq L$ consist of the cubes colored green, and the cubes on the right incident to red edges are the neighbors of $S$ in the final product.
%     The two bipartite base graphs $G_L, G_R$ have the structure that $M$ has $k$ parts, and for $u\in M$ and $v,w \in M$ from a different part, the common neighborhood $N_{G_R}(u) \cap N_{G_R}(v)$ and $N_{G_R}(u) \cap N_{G_R}(w) \subseteq R$ are disjoint, each corresponding to a \emph{special set} in $H_u$ (the copy of $H$ placed on $u$).
%     The middle-to-right analysis involves upper bounding the collisions of the red edges on the right.
%     In the figure, $u$ and $v$ have $2$ collisions. We will show that this cannot happen too often.}
%     \label{fig:red-edges}
% \end{figure}

In the remainder of the argument, we prove that the collisions between neighborhoods of different gadgets inflict negligible damage on expansion.

We next show that the red edges have few collisions in $R$.
We will crucially use the small-set skeleton expansion with $\lambda = O(D^{1/4})$ and the special set structure of $G_R$ with $s = \Theta(\sqrt{D})$ (\Cref{def:structured-bipartite-graph,lem:base-graph}).

We construct the \emph{collision graph} $C$ --- the multi-graph $C$ on vertex set $U \subseteq M$ by placing a copy of the edge $\{u,v\}$ for each $u \neq v \in U$, and $r\in R$ such that $\{u,r\}$ and $\{v,r\}$ are red edges in $\RED$.
See \Cref{fig:collision-graph} for an example.
The number of neighbors of $S$ in the final product $Z$ is at least
\begin{align*}
    e(\RED) - e(C) \mcom
\end{align*}
since a vertex $v \in R$ with degree $d_v$ in $\RED$ contributes one neighbor, but it is counted $d_v$ times in $e(\RED)$ and $\parens*{d_v \atop 2}$ times in $e(C)$, and $d_v - \parens*{d_v \atop 2} \leq 1$ for all $d_v \in \N$.

We will need the following folklore fact.
\begin{lemma}[{\cite[Lemma 2.17]{HLMOZ25}}]
\label{lem:out-degree-bound}
    Given a graph $G$ whose adjacency matrix has maximum eigenvalue $\lambda$, then there is an orientation of the edges in $G$ such that all vertices have out-degree at most $\lambda$.
\end{lemma}

% \begin{lemma}[{\cite[Lemma 6.2]{HMMP24}}] \label{lem:average-degree-bound}
%     Let $G$ be a bipartite graph with average left-degree $d_1$ and average right-degree $d_2$.
%     Let $\lambda$ be the maximum eigenvalue of the adjacency matrix.
%     Then, $(d_1-1)(d_2-1) \leq \lambda^2$.
% \end{lemma}

\begin{claim} \label{claim:middle-to-right}
    Suppose $k\delta^2 \leq o_D(1)$, $\lambda \leq s \delta$, and $d_L \geq \frac{1}{\delta} \max\{\lambda, \sqrt{s}\} \log D$.
    Then, $e(C) \leq o_D(1) \cdot kd_L|S|$.
\end{claim}
\begin{proof}
    Let $\ul{C}$ be the simple graph obtained by removing duplicate edges from $C$.
    Moreover, let $\wt{G}_R$ be the simple graph on $M$ where $u \neq v\in M$ are connected if they have a common neighbor in $R$ in the graph $G_R$.
    Clearly, $\ul{C}$ is a subgraph of $\wt{G}_R$.
    Moreover, recall from \Cref{def:structured-bipartite-graph} that $M$ is a union of $k$ vertex sets,
    and thus $\wt{G}_R$ is $k$-partite. Let us now restrict $C$ to edges between two parts $a,b\in [k]$.
    We will write $r = r(a,b)$ and the special sets $Q_i = Q_i^{a,b}$ for simplicity.
    
    By the $\lambda$-small set skeleton expansion, we have that $\ul{C}$ has largest eigenvalue at most $\lambda$.
    This intuitively means that $\ul{C}$ contains very few edges.
    Next, we need to upper bound the multiplicities of edges in $C$.
    The main observation is that if $u \in M_a$ and $v \in M_b$ are connected in $\wt{G}_R$, then $u, v$ in fact have many common neighbors in $G_R$.
    More specifically, $u$ has neighbors $v_1, v_2,\dots,v_r$ in $\wt{G}_R$, and each common neighborhood $N_{G_R}(u) \cap N_{G_R}(v_i) \subseteq R$ corresponds to a special set as in \Cref{def:structured-bipartite-graph}.
    On the other hand, the pseudorandomness of the gadget $H$ implies that the red edges coming out of $u$ must be evenly spread among the special sets.
    In the following, we make this intuition formal.

    The largest eigenvalue of $\ul{C}$ is at most $\lambda$.
    Thus, by \Cref{lem:out-degree-bound}, there is an orientation of the edges of $\ul{C}$ such that all vertices have out-degree at most $\lambda$.
    Pick such an orientation, and let $\Out(u)$ be the set of out-going edges incident to $u$.
    Then,
    \begin{align*}
        e(C) = \sum_{u\in U} \sum_{e\in \Out(u)} \mathrm{multiplicity}(e) \mper
    \end{align*}
    Due to the special set structure of $G_R$ (\Cref{def:structured-bipartite-graph}), for any $u \in M_a$ and $v_1,\dots,v_r$ (potentially) connected in $\ul{C}$, their common neighborhoods within $G_R$ are exactly special sets in the gadget $H_u$ --- that is, $N_{G_R}(u) \cap N_{G_R}(v_i) = \RNbr_u(Q_i)$, and each $|Q_i| \in \left[\frac{D_R}{2s}, \frac{2D_R}{s}\right]$ where $s = \Theta(\sqrt{D})$ from \Cref{lem:base-graph}.

    Thus, we can upper bound $\sum_{e\in \Out(v)} \mathrm{multiplicity}(e)$ by the number of red edges that land in any $|\Out(v)|$ of the special sets.
    Denote $\deg_S(v) \coloneqq \deg_{\GLSU}(v)$.
    By \Cref{lem:pr-gadget}, applying the bound with $|W| = \max\braces*{|\Out(v)|, \frac{s\log D}{d_L}} \leq \max\braces*{\lambda, \frac{s\log D}{d_L}}$ and $|A| = \deg_S(v)$, we get
    \begin{align*}
        \sum_{e\in \Out(v)} \mathrm{multiplicity}(e) 
        &\le O(1) \cdot \max \braces*{\lambda,\ \frac{s \log D}{d_L}} \cdot \max\braces*{ \frac{d_L}{s}\cdot \deg_S(v),\ \log D } \\
        &\le O(1) \cdot \max\braces*{ \frac{\lambda}{s},\ \frac{\lambda \log D}{d_L \deg_S(v)},\ \frac{\log D}{d_L},\ \frac{s \log^2 D}{d_L^2 \deg_S(v)} }
        \cdot d_L \cdot \deg_S(v) \\
        &\leq O(\delta) \cdot d_L \cdot \deg_S(v) \mper
    \end{align*}
    Here, we use the assumptions on the parameters: $\lambda \leq \delta s$, and $d_L \geq \frac{1}{\delta}\max\{\lambda, \sqrt{s}\} \log D \geq \frac{1}{\delta}\log D$.
    
    Summing over $v\in U$, we get
    \begin{align*}
        e(C) &\leq O(\delta) \cdot d_L \sum_{v\in U} \deg_S(v) 
        \leq O(\delta) \cdot kd_L |S| \mper
    \end{align*}
    The above is restricted to one pair $a,b\in [k]$.
    For the final bound, we multiply the above by $k^2$.
    Since $k^2\delta \leq o_D(1)$, we get $e(C) \leq o_D(1) \cdot kd_L|S|$.
\end{proof}

Finally, we combine the above to finish the proof of \Cref{thm:main}.
With $\delta \leq o_D(1) \cdot \frac{1}{k^2}$ and $k \geq 16/\eps^2$,
\Cref{claim:left-to-middle} and \Cref{eq:red-edges} imply that
\begin{align*}
    e(\RED) \geq (1 - o_D(1)) \cdot d_L \cdot \parens*{1-\sqrt{\delta}- 2k^{-1/2} } k|S|
    \geq (1 - \eps/2) kd_L|S|\mper
\end{align*}
The number of neighbors of $S$ in the final product $Z$ is at least $e(\RED) - e(C)$, and by \Cref{claim:middle-to-right} we have $e(C) \leq o_D(1) \cdot kd_L|S|$.
Thus, choosing $D$ large enough,
\begin{align*}
    |N_Z(S)| \geq (1 - \eps) kd_L|S| \mper
\end{align*}
The analysis for the expansion of any $T \subseteq R$ is identical.
This finishes the proof.

%% file: cubical.tex
\section{Cubical complexes and coded incidence graphs} \label{sec:cubical-complexes}

\parhead{Notation and terminology.}
Given subsets $A,B$ of a group $\Gamma$ with multiplication operation $\cdot$, we define $A\cdot B$ to refer to the product set $\{a\cdot b : a\in A, b\in B\}$.

We start with the definition of cubical generating sets.

\begin{definition}[Cubical generating set]
    Let $\Gamma$ be a finite group and $k\in\bbN$.
    We say $A_1,A_2,\dots,A_k\subseteq \Gamma$ are \emph{cubical generating sets} if they are closed under inverses, and
    \begin{itemize}
        \item $A_i\cdot A_j = A_j \cdot A_i$ for all $i \ne j$,
        \item $\abs*{A_1\cdots A_k} = \abs*{A_1} \cdots \abs*{A_k}$.
    \end{itemize}
\end{definition}

\begin{definition}[Decorated Cayley cubical complex]
\label{def:cubical-complex}
    Given a finite group $\Gamma$ and cubical generating sets $\calA = (A_1,\dots,A_k)$, the \emph{(decorated) Cayley cubical complex} $X = \Cay(\Gamma; \calA)$ is defined by:
    \begin{itemize}
        \item its vertex set $X(0) = \Gamma \times \zo^k$,
        \item its $k$-face set $X(k)$ consisting of all $2^k$-sized subsets of $X(0)$ of the form $f = \{(f_x, x)\}_{x\in\zo^k}$ such that for every edge $\{x, x \oplus e_i\}$ of the hypercube, $f_x^{-1} f_{x\oplus e_i} \in A_i$.
        \item For $I \subseteq [k]$, we define an $I$-subcube to be all $\zo^k$ strings of the form $y \oplus \bigoplus_{i \in I} b_i e_i$, where $b_i \in \{ 0, 1 \}$ and $e_i$ denotes the vector with a $1$ in the $i$'th index. The dimension of an $I$-subcube is $|I|$.  
        \item For a subcube $C$ of $\zo^k$, we define the set of $C$-faces $X(C)$ as:
        \[
            X(C) \coloneqq \braces*{ \braces*{(f_x, x)}_{x\in C} : f\in X(k)}\mper
        \]
        We define the set of $i$-faces as $X(i) \coloneqq \bigcup_{C:\dim(C)=i} X(C)$.
    \end{itemize}
\end{definition}

We use the word ``decorated'' since the vertex set $X(0)$ consists of $2^k$ copies of $\Gamma$, as opposed to the usual way of Cayley graphs on $\Gamma$.

Henceforth, we fix a group $\Gamma$ along with cubical generating sets $A_1,\dots,A_k$, and let $X = \Cay(\Gamma; (A_1,\dots,A_k))$.

One important property of cubical complexes is that for any two points $(g, \vec{0})$ and $(g', \vec{1})$ in opposite corners, there is at most one $k$-face $f\in X(k)$ that contains the two points.
More generally, given $U = \{(g_1, \vec{0}), (g_2,x_2),\dots,(g_m, x_m)\}$, any face restricted to the subcube of the coordinates $\bigcup_{t>1} \supp(x_t)$ is uniquely identified (if exists).
An example is given in \Cref{fig:cubical-complex}. The points $(g, 000)$ and $(g a_1 a_2 a_3, 111)$ uniquely identify a $3$-face.
Moreover, the points $(g, 000)$, $(ga_1a_2, 110)$ and $(ga_1 a_3', 101)$ also uniquely identify a $3$-face, since $\supp(110) \cup \supp(101) = [3]$.

This property is crucial in our construction, and a more general form is formalized in the following lemma.

% We now establish some basic facts about Cayley cubical complexes.
\begin{lemma} \label{lem:cube-count}
    For any $U\subseteq X(0)$ where $U = \{(g_1,x_1),\dots,(g_m,x_m)\}$, define
    \[
        S(U) = \braces*{i\in[k]:\exists\, s,t\in[m]\text{ s.t. }x_s[i] \ne x_t[i] }
        = \bigcup_{t > 1} \supp(x_t \oplus x_1) \mcom
    \]
    and subcube
    \[
        C(U) = x_1 \oplus \bigoplus_{i \in S(U)} \{0,1\} \cdot e_i \mper
    \]
    There is at most one $C(U)$-face containing $U$, and if such an $C(U)$-face exists, the number of $k$-faces containing $U$ is equal to $\prod_{i\notin S(U)} |A_i|$.
\end{lemma}
\begin{proof}
    We will first prove that there is at most one $C(U)$-face containing $U$, and then prove that if nonzero, the number of $k$-faces containing $U$ is equal to $\prod_{i\notin S(U)} \abs*{A_i}$.

    \parhead{Proof that there is at most one $C(U)$-face containing $U$.}
    Define $B_{r}^S(x)$ as the set of all vectors $y$ in $\zo^k$ such that the Hamming weight of $x\oplus y$ is at most $r$ and $\supp(x\oplus y) \subseteq S$.
    We will prove for every $r\ge 1$ and each $y\in B_{r}^{S(U)}(x_1)$, there exists an element $g_y \in \Gamma$ such that $f_y = g_y$ for every face $f$ containing $U$.
    Indeed, this claim implies that there can be at most one $C(U)$-face containing $U$.

    We start by proving the claim for $r = 1$.
    Let $y = x_1 \oplus e_i \in B_1^{S(U)}(x_1)$ where $i\in S(U)$.
    Note that $i\in S(U)$ means that 
    there is a $t\in[m]$ such that $x_1[i] \ne x_t[i]$.
    We will prove that the points $(g_1, x_1)$ and $(g_t, x_t)$ uniquely identify $(f_y, y)$.
    Equivalently, any pair of faces $f$ and $f'$ containing $U$ must have $f_y = f'_y$.
    
    Define $a_i = g_1^{-1} f_y$ and $a_i' = g_1^{-1} f_y'$.
    Note that both $a_i$ and $a_i'$ must be in $A_i$.
    Pick an arbitrary order $j_1,\dots,j_{\ell}$ for the coordinates in $\supp(x_1\oplus x_t) \setminus \{i\}$.
    Next, observe that the sets $E \coloneqq a_i\cdot A_{j_1} \cdots A_{j_{\ell}}$ and $E' \coloneqq a_i' \cdot A_{j_1} \cdots A_{j_{\ell}}$, which both have size $\abs*{A_{j_1}}\cdots\abs*{A_{j_{\ell}}}$, must have a nonempty intersection since they both must contain $g_1^{-1} g_t$.
    Now, $\abs*{A_i \cdot A_{j_1} \cdots A_{j_{\ell}} } = \abs*{A_i}\cdot \abs*{A_{j_1}} \cdots \abs*{A_{j_{\ell}}}$, and thus if $a_i \ne a_i'$, then $E$ and $E'$ must be disjoint.
    Therefore, $a_i = a_i'$ and $f_y = f_y'$.

    For the inductive step, assume that for some $r\ge 2$, the uniqueness statement holds for all $y\in B_{r-1}^{S(U)}(x_1)$.
    Let $f$ be any face containing $U$ and let $y \in B_{r}^{S(U)}(x_1)$.
    We will prove that $f_y$ is uniquely determined.
    Define $U' \coloneqq U \cup \braces*{ (g_x, x) : x \in B_{r-1}^{S(U)}(x_1) }$ where $g_x$ is the unique value of $f_x$ for any face $f$ containing $U$.
    Note that $S(U') = S(U)$.
    Observe that $\supp(y\oplus x_1)$ is nonempty by the assumption that $r \ge 2$, and let $i$ be an arbitrary element contained within.
    This means that $y \oplus e_i \in B_{r-1}^{S(U)}(x_1)$.
    Since $S(U') = S(U)$, the conclusion that $f_y$ is uniquely determined follows by applying the statement we established for $r = 1$ to $U'$ in place of $U$ and $y\oplus e_i$ in place of $x_1$.

    \parhead{On number of ways to extend a $C(U)$-face to a $k$-face.}
    It remains to prove that the number of ways to extend a $C(U)$-face to a full $k$-face is equal to $\prod_{i\notin S(U)} \abs*{A_i}$.
    To this end, fix an order $i_1,\dots,i_{\ell}$ of coordinates in $\ol{S(U)}$ arbitrarily.
    For each choice of $(a_i \in A_i)_{i\notin S(U)}$, we will prove that there is a unique $k$-face $f$ containing $U\cup\braces*{ \parens*{g_1\cdot a_{i_1}\cdots a_{i_{\ell}}, x_1 \oplus 1_{\ol{S(U)}} } }$.
    The conclusion will follow from the fact that there are $\prod_{i\notin S(U)} \abs*{A_i}$ many choices for $(a_i)_{i\notin S(U)}$.

    We will construct this face $f$ by describing $f_y$ for each $y\in\zo^k$.
    We will first treat the case of $y$ of the form $x_1\oplus\Delta$ for $\Delta$ supported on coordinates outside $S(U)$.
    Let $j_1,\dots,j_s$ be the coordinates in the support of $\Delta$, and let $j'_1,\dots,j'_{\ell-s}$ be an arbitrary order for coordinates in $\{i_1,\dots,i_{\ell}\}\setminus\{j_1,\dots,j_s\}$.
    Now, by the property that $A_i\cdot A_j = A_j \cdot A_i$ for every $i,j$, we have:
    \[
        g_1\cdot a_{i_1}\cdots a_{i_{\ell}} = g_1 \cdot a'_{j_1}\cdots a'_{j_s} \cdot a'_{j'_1} \cdots a'_{j'_{\ell-s}}
    \]
    where $a'_j \in A_j$.
    We define $f_y$ as $g_1\cdot a'_{j_1}\cdots a'_{j_s}$.

    We now construct $f_y$ for general $y\in\zo^k$.
    Observe that $y$ can be written as $z\oplus\Delta$ for $z\in C(U)$ and $\Delta$ supported only on coordinates outside $S(U)$.
    Let $j_1,\dots,j_s$ be the coordinates in the support of $\Delta$, and let $j'_1,\dots,j'_{s'}$ be the coordinates in the support of $x_1\oplus z$.
    Now, we can write:
    \begin{align*}
        f_{x_1\oplus\Delta} &= g_1 \cdot a'_{j_1}\cdots a'_{j_s} \\
        &= g_z \cdot a'_{j'_1} \cdots a'_{j'_{s'}} \cdot a'_{j_1}\cdots a'_{j_s} \\
        &= g_z \cdot a''_{j_1} \cdots a''_{j_s} \cdot a''_{j'_1} \cdots a''_{j'_{s'}}\mcom
    \end{align*}
    where $a''_j\in A_j$.
    In the above, we used the construction of $f_{x_1\oplus\Delta}$ from earlier in the first equality, the fact that there is a $C(U)$-face containing $(g_z,z)$ and $(g_1,x_1)$ in the second equality, and $A_i\cdot A_j = A_j \cdot A_i$ in the third equality.
    Finally, we set $f_y$ as $g_z\cdot a''_{j_1} \cdots a''_{j_s}$.
    It can easily be checked using the set-commuting relation that $f$ is indeed a valid $k$-face.
    Finally, $f$ is the unique face containing $\wt{U} \coloneqq U\cup\braces*{ \parens*{g_1\cdot a_{i_1}\cdots a_{i_{\ell}}, x_1 \oplus 1_{\ol{S(U)}} } }$ since $S\parens{\wt{U}} = [k]$, which completes the proof.
\end{proof}

Finally, we define a natural notion of expansion in a cubical complex that is useful for our purposes.
\begin{definition}[Expanding cubical complex]
\label{def:expanding-cubical-complex}
    We say that a cubical complex $X = \Cay(\Gamma; (A_1,\dots,A_k))$ is $\alpha$-expanding if for any $x,y\in \zo^k$, the bipartite graph $\calI_{y,y\oplus x}$ with edge set $\braces*{ \braces*{(g, y), (g \cdot \prod_{i=1}^k a_i^{x_i}, y\oplus x)} : g\in \Gamma, a_i\in A_i }$, which has degree $d_x(X) = \prod_{i=1}^k |A_i|^{x_i}$, has second eigenvalue at most $\alpha \sqrt{d_x(X)}$.
    For $i\in[k]$, we define
    \(
        d_i(X) \coloneqq \max_{x\in\zo^k:\,|\supp(x)| = i} d_x(X)\mper
    \)
\end{definition}

The following theorem is essentially contained in \cite{RSV19} in a different form.
We provide a mostly self-contained proof in \Cref{sec:cubical-construction}, assuming only that the expander graphs of Lubotzky--Phillips--Sarnak \cite{LPS88} are Ramanujan.
\begin{theorem}    \label{thm:cubical-complexes}
    Let $p_1 < \dots < p_k$ and $q > 2\sqrt{\prod_{i=1}^k p_i}$ be any prime numbers congruent to $1$ mod $4$, and each $p_i$ is a quadratic residue modulo $q$.
    There is an explicit choice of cubical generating sets $A_1,\dots,A_k$ on $\Gamma = \PSL_2(\bbF_q)$ such that $|A_i| = p_i+1$ and the cubical complex $X=\Cay(\Gamma; (A_1,\dots,A_k))$ is $2^k$-expanding.
\end{theorem}

\parhead{Base graph construction.}
We will construct our bipartite base graph based on a cubical complex $X$ and a code $\calC \subseteq \zo^k$.
To do so, we first introduce the notion of the ``signature'' of a cube.

\begin{definition}[Signature of cube]
    Given a $k$-face $f\in X(k)$, its \emph{signature} is the following labeling of the directed edges of the $k$-dimensional hypercube with elements of $\Gamma$:
    for every $x\in\zo^k$ and every $i\in[k]$, we label the directed edge $(x,x\oplus e_i)$ with $f_x^{-1} f_{x\oplus e_i}$.
\end{definition}

\begin{definition}[Coded cubical incidence graph]
\label{def:cubical-incidence-graph}
    Given a code $\calC \subseteq \zo^k$, the $\calC$-\emph{cubical incidence graph} of a cubical complex $X$ is the edge-labeled bipartite graph $(V_1, V_2, E)$ such that $V_1 = X(k)$, $V_2 = \Gamma \times \calC \subseteq X(0)$, and $f\in X(k)$ and $(g,x)\in V_2$ are connected iff $(g,x)\in f$.
    Further, an edge between $f$ and $(g,x)$ is labeled with the signature of $f$. 
\end{definition}

Our construction uses the cubical incidence graph arising from the Hadamard code, of which we use minimal properties.
\begin{fact}    \label{fact:hadamard}
    Let $k$ be a power of $2$.
    The $k$-th Hadamard code $\calH_k$ is a linear code in $\F_2^k$ of dimension $\log_2 k$ where for all distinct $x,y\in\calH_k$, the Hamming distance between $x$ and $y$ is exactly $k/2$.
\end{fact}

\begin{remark}
\label{rem:balanced-code}
    For our purposes, any linear code with dimension growing in $k$ and pairwise distance between $\frac{2}{5}+\delta$ and $\frac{3}{5}-\delta$ would suffice.
    The rate and distance of the chosen code determine the trade-off between the degree $d$ and the parameter $\eps$ in the $(1-\eps)$-vertex expansion.
    However, we do not optimize this dependence and use the Hadamard code for simplicity.
\end{remark}

\subsection{Proof of \texorpdfstring{\Cref{lem:base-graph}}{Lemma~\ref{lem:base-graph}}: structured bipartite graph construction} \label{sec:proof-lem-base-graph}

We now construct structured bipartite graphs (\Cref{def:structured-bipartite-graph}) with the parameters specified in \Cref{lem:base-graph}.
It is quite straightforward to see that the $\calC$-cubical incidence graph of a cubical complex from \Cref{thm:cubical-complexes} has the desired special set structure and small-set skeleton expansion, while we defer the proof of small-set $2\sqrt{k}$-neighbor expansion to \Cref{sec:subcube-density}.
However, since the construction from \Cref{thm:cubical-complexes} restricts the degrees to be products of primes, we must remove some faces according to their signatures to get the desired degrees $D_L, D_R$.

We will need the following folklore fact (see, e.g., \cite[Lemma 3.13]{HLMOZ25} for a proof).
\begin{lemma}   \label{lem:spectral-to-skeleton}
    For any $n$-vertex $d$-regular graph $G$ with largest nontrivial eigenvalue $\lambda$, and any subgraph $H$ of $G$ incident to at most $\delta n$ vertices, the largest eigenvalue of $H$ is at most $\lambda + \delta d$.
\end{lemma}

Let $p_1,\dots,p_k$ and $p_1',\dots,p_k'$ be $2k$ distinct primes congruent to $1$ mod $4$ such that each $D^{1/k}\le p_i \le 2 D^{1/k}$, and let $q$ be a prime of the form $1+4\ell\prod_{i=1}^k p_i p_i'$ for $\ell\in\bbN$.
These primes exist due to \Cref{fact:prime-density}.
Let $X$ be the cubical complex given by \Cref{thm:cubical-complexes} for $p_1,\dots,p_k$ and $q$, and let $X'$ be the corresponding cubical complex for $p_1',\dots,p_k'$ and $q$.
Let $\calC = \calH_k \subseteq \F_2^k$ be the Hadamard code, let $\ul{D}_L \coloneqq \prod_{i=1}^k (p_i+1)$, and let $\ul{D}_R \coloneqq \prod_{i=1}^k (p_i'+1)$.
Finally, let $\ul{G}_L = (\ul{L}, M, E_{\ul{L}})$ and $\ul{G}_R = (\ul{R}, M, E_{\ul{R}})$ be the $\calC$-cubical incidence graphs (\Cref{def:cubical-incidence-graph}) of $X$ and $X'$ respectively.

We first prove the desired properties for $\ul{G}_L$ and $\ul{G}_R$, and then show how to construct $G_L$ and $G_R$ from them, which inherit the desired properties and additionally are $(k,D_L)$-biregular and $(k,D_R)$-biregular respectively.

\parhead{Small-set skeleton expansion.}
Recall that $M = \Gamma \times \calC$ has $|\calC| = k$ parts,
and the skeleton of $X$ (\Cref{def:skeleton-expansion}) is the simple graph on $M$ where vertices $(g, x), (h, y) \in M$ are connected if they are contained in some face $f \in X(k)$.
Thus, the skeleton of $X$ is the union of bipartite graphs over each pair $x \neq y \in \calC$ with edges $\{(g,x), (g \cdot \prod_{i=1}^k a_i^{x_i \oplus y_i}, y) \}$ for $g\in \Gamma$ and $a_i \in A_i$.
Since $x,y$ have distance exactly $k/2$, the degree of the bipartite graph is $d_{x\oplus y} = \prod_{i=1}^k |A_i|^{x_i \oplus y_i} = O(\sqrt{D})$.
By the fact that $X$ is $2^k$-expanding (from \Cref{thm:cubical-complexes}), its second eigenvalue is at most $2^k \sqrt{d_{x\oplus y}} \leq O(D^{1/4})$.
By \Cref{lem:spectral-to-skeleton}, we get that $\ul{G}_L$ is an $O\parens*{D^{1/4}}$-skeleton expander.
The same argument applies for $\ul{G}_R$.

\parhead{Bound on the number of special sets.}
For every $x,y\in\calC$, along with any signature $\sigma$ on the subcube given by $C_{x,y}\coloneqq\{x\oplus z:\supp(z)\subseteq\supp(x\oplus y)\}$, let $Q_{\sigma}$ be the set of all signatures $\tau$ of the hypercube that extend $\sigma$.
The number of choices of $x,y$ and signature $\sigma$ on the subcube is at most $k^2\cdot \sqrt{D}$.
It can be verified that for any pair of vertices $u,v$, either the neighborhoods are empty, or are described by one of the sets $Q_{\sigma}$.

\parhead{Small-set $2\sqrt{k}$-neighbor expansion.}
The precise statement from which our bounds on small-set $2\sqrt{k}$-neighbor expansion follows is given below.
\begin{lemma}   \label{lem:small-set-subcube-expansion}
    For any subset of vertices $U\subseteq M$ of size at most $D^{-1}|M|$, we have that the number of vertices in $\ul{L}$ and $\ul{R}$ with more than $2\sqrt{k}$ neighbors in $U$ is at most $O\parens*{D^{5/8}}|U|$.
\end{lemma}
We defer the proof of \Cref{lem:small-set-subcube-expansion} to \Cref{sec:subcube-density}, and describe how to construct $G_L$ and $G_R$.

\parhead{Satisfying degree constraints.}
There is a collection $\ul{\calS}_L$ of $\ul{D}_L$ distinct signatures $\tau$ such that every $m\in M$ is incident to exactly one element of $\ul{L}$ with signature $\tau$ in $\ul{G}_L$.
Likewise, there is a collection $\ul{\calS}_R$ of $\ul{D}_R$ distinct signatures $\tau$ such that every $m\in M$ is incident to exactly one element of $\ul{R}$ with signature $\tau$ in $\ul{G}_R$.

We pick an arbitrary $D_L$-sized subcollection $\calS_L$ of $\ul{\calS}_L$ and an arbitrary $D_R$-sized subcollection $\calS_R$ of $\ul{\calS}_R$, and define $L$ and $R$ as:
\[
    L\coloneqq\braces*{ v\in\ul{L}: \mathrm{Signature}(v) \in \calS_L } \mcom \qquad R \coloneqq \braces*{ v\in\ul{R}: \mathrm{Signature}(v) \in \calS_R }\mper
\]
We now define $G_L$ and $G_R$ as the induced subgraphs $\ul{G}_L[L,M]$ and $\ul{G}_R[R,M]$ respectively.
The graphs $G_L$ and $G_R$ are $(k, D_L)$- and $(k, D_R)$-biregular bipartite graphs, respectively, and each inherits the desired small-set skeleton expansion and small-set $2\sqrt{k}$-neighbor expansion properties from its parent graph.

\parhead{Neighborhood functions.}
Arbitrarily order the $D_L$ signatures in $\calS_L$ as $\ell_1,\dots,\ell_{D_L}$, and the $D_R$ signatures in $\calS_R$ as $r_1,\dots,r_{D_R}$.
For any vertex $u\in M$ and $i\in\bracks*{D_L}$, the function $\LNbr_u(i)$ maps to the neighbor of $u$ in $L$ with the signature $\ell_i$, and similarly for $i\in\bracks*{D_R}$, $\RNbr_u(i)$ maps to the neighbor of $u$ with signature $r_i$.
\qed

%% file: left-to-middle.tex
\subsection{Small-set subcube density in cubical complexes} \label{sec:subcube-density}

In this section, we prove \Cref{lem:small-set-subcube-expansion}, which states that for any small enough subset $U \subseteq M = \Gamma \times \calH_k$, there are at most $O_k(D^{5/8}) |U|$ faces $f\in X(k)$ that contain at least $2\sqrt{k}$ vertices in $U$.
Here, recall that $\calH_k \subseteq \F_2^k$ is the $k$-th Hadamard code of distance $k/2$ (\Cref{fact:hadamard}).
Thus, the following lemma directly implies \Cref{lem:small-set-subcube-expansion}.

\begin{lemma}   \label{lem:small-set-subcube-density}
    Let $\Gamma$ be a group with cubical generating sets $A_1,\dots,A_k$ such that $\max_{i\in[k]}\abs*{A_i} \le 2 \cdot\min_{i\in[k]}\abs*{A_i}$.
    Let $D\coloneqq \prod_{i\in[k]}\abs*{A_i}$, and
    let $X = \Cay(\Gamma; (A_1, \dots, A_k))$ be a $2^k$-expanding cubical complex with vertex set $X(0) = \Gamma \times \F_2^k$.
    Then, for any $U\subseteq \Gamma \times \calH_k$ where $|U| \le D^{-1} |\Gamma \times \calH_k|$, 
    we have:
    \[
        \abs*{ \braces*{ f \in X(k) : |f \cap U| \ge 2\sqrt{k} } } \le O_k\parens*{D^{5/8}}\cdot |U|\mper
    \]
\end{lemma}

\parhead{Notations.}
For a vertex $(g,s)\in X(0)$, we say that it has \emph{type} $s \in \F_2^k$.
We use $F_k(U; \ge 2\sqrt{k})$ to denote the set of $k$-faces $\braces*{ f \in X(k) : |f \cap U| \ge 2\sqrt{k} }$, which is what we will bound in \Cref{lem:small-set-subcube-density}.
More generally, for $\sigma \subseteq \calH_k$, we define $F_k(U; \sigma)$ to be the set of all $k$-faces whose vertices with types in $\sigma$ lie in $U$, i.e., $F_k(U;\sigma) \coloneqq \braces*{f\in X(k): (f_s, s)\in U,\ \forall s\in \sigma}$.
When restricterd to a subcube $C \subseteq \F_2^k$, we use $F_C(U; \sigma)$ to denote the $C$-faces in $X(C)$ (recall \Cref{def:cubical-complex}) whose vertices with types in $\sigma$ lie in $U$.

\medskip

Our first observation is that for any $f \in F_k(U; \ge 2\sqrt{k})$, $f \cap U$ must contain four vertices whose types sum to $0$.

\begin{lemma} \label{lem:s1-s2-s3-s4}
    Let $S \subseteq \calH_k$ be of size $\ge 2\sqrt{k}$. Then there exists a four-tuple of distinct elements $\sigma \in S^4$ for which $\sigma_1 \oplus \sigma_2 \oplus \sigma_3 \oplus \sigma_4 = 0$.
\end{lemma}

\begin{proof}
    Consider the set of sums of two distinct elements of $S$. Since there are $\parens*{{|S| \atop 2}} \ge \parens*{{2\sqrt{k}\atop 2}} > k$ such sums, whereas there are only $|\calH_k| = k$ possible values for the sum, there must be two distinct pairs of elements that have the same sum. Namely, there are elements $\sigma_1, \sigma_2, \sigma_3, \sigma_4 \in S$ for which $\sigma_1 + \sigma_2 = \sigma_3 + \sigma_4$.
    Note that $\sigma_1, \sigma_2, \sigma_3, \sigma_4$ must be pairwise distinct: if for instance $\sigma_1 = \sigma_3$, then $\sigma_2 = \sigma_4$ also, which implies that the pair $\{\sigma_1, \sigma_2\}$ is equal to the pair $\{\sigma_3, \sigma_4\}$.
\end{proof}

We may therefore partition the set $F_k(U; \ge 2\sqrt{k})$ according to the value of the four vertex types that sum to $0$. In particular, $F_k(U; \sigma)$ is the set of all $k$-faces that have four vertices of types $\sigma_1, \sigma_2, \sigma_3, \sigma_4$ in $U$. Then 
\[
    F_k(U; \ge 2\sqrt{k}) \subseteq \bigcup_{\sigma :\, \sigma_1\oplus\sigma_2\oplus\sigma_3\oplus\sigma_4 = 0} F_k(U; \sigma),
\]
which lets us bound $|F_k(U; \ge 2\sqrt{k})|$ by 
\[
    |F_k(U; \ge 2\sqrt{k})| \le \sum_{\sigma :\, \sigma_1\oplus\sigma_2\oplus\sigma_3\oplus\sigma_4 = 0} |F_k(U; \sigma)|.  \numberthis \label{eq:final-bound}
\]
It therefore suffices to upper bound the size of each $F_k(U; \sigma)$ individually.

To this end, fix $\sigma \in \calH_k^4$ for which $\sigma_1\oplus\sigma_2\oplus\sigma_3\oplus\sigma_4 = 0$.
The tuple $\sigma$ determines a subcube 
\[
    C_{\sigma} = \sigma_1 \oplus \bigoplus_{i \in \Delta(\sigma)} \{ 0, 1 \} \cdot e_i \mcom
    \numberthis \label{eq:subcube-sigma}
\] 
where 
\[
    \Delta(\sigma) \coloneqq \braces*{ i \in [k] : \exists j_1, j_2 \in [4] \text{ s.t. } \sigma_{j_1}[i] \neq \sigma_{j_2}[k]}
    = \bigcup_{j \in \{2,3,4\}} \supp(\sigma_1 \oplus \sigma_j) \mper
\]
Let us establish some properties of $\Delta(\sigma)$.

\begin{claim} \label{claim:abc}
    For any $(\sigma_1, \sigma_2, \sigma_3, \sigma_4) \in \calH_k^4$ that sum to $0$ over $\F_2^k$, there are three disjoint sets $a,b,c \subseteq [k]$, each of size $k/4$, for which 
    \begin{align*}
        \supp(\sigma_1 \oplus \sigma_2) &= a \cup b \\
        \supp(\sigma_1 \oplus \sigma_3) &= a \cup c \\
        \supp(\sigma_1 \oplus \sigma_4) &= b \cup c \mper
    \end{align*}
    In particular, $\Delta(\sigma) = a \cup b \cup c$ is of size $3k/4$.
\end{claim}

\begin{proof}    
    Notice that $\sigma'_2 \coloneqq \sigma_2 \oplus \sigma_1$ and $\sigma'_3 \coloneqq \sigma_3 \oplus \sigma_1$ are distinct codewords of $\calH_k$, and hence have weight $k/2$. Furthermore, the distance between $\sigma'_2$ and $\sigma'_3$ is also $k/2$. Define $a = \supp(\sigma'_2) \cap \supp(\sigma'_3)$. Then, the Hamming distance between $\sigma'_2$ and $\sigma'_3$, which is $k/2$, can also be written as $(k/2 - |a|) + (k/2 - |a|)$, implying that $|a| = k/4$. We can now define $b = \supp(\sigma'_2) \backslash a$ and $c = \supp(\sigma'_3) \backslash a$, which will both be of size $k/4$ as well. We simply need to check that $\supp(\sigma_4 \oplus \sigma_1) = b \cup c$, which we do as follows: $\sigma_4 \oplus \sigma_1 = \sigma_2 \oplus \sigma_3 = \sigma'_2 \oplus \sigma'_3$ implies that $\supp(\sigma_4 \oplus \sigma_1) = \supp(\sigma'_2 \oplus \sigma'_3) = b \cup c$.
\end{proof}

For any element $x\in\calH_k$, we use $U_{x}$ to denote $U\cap (\Gamma \times \{x\})$.
For a subcube $C$ of $\zo^k$, recall that $F_{C}(U; \sigma)$ is all $C$-faces with a vertex in each $U_{\sigma_i}$ for $\sigma_i\in\sigma$.
By \Cref{lem:cube-count}, each $f' \in F_{C_{\sigma}}(U; \sigma)$ can be extended to a $k$-face $f \in F_k(U;\sigma)$ in $\prod_{i \not\in \Delta(\sigma)} \abs*{A_i}$ ways.

In the remainder of this section, we will use $C$ to refer to $C_{\sigma}$.
We can further partition $F_{C}(U; \sigma)$ based on the value of its type-$\sigma_1$ vertex.
That is, for $u \in U_{\sigma_1}$, define
\[
    F_{C}(u;U;\sigma) \coloneqq \{ f \in F_{C}(U; \sigma) : u \in f \} \mper
\]
We will bound the size of $F_{C}(u;U;\sigma)$ in the following lemma. 

In order to state the bound, we define the $s$-neighborhood $N_{s}(u)$ of $u \in U_{\sigma_1}$, for $s\in \F_2^k$, as all the neighbors of $u$ in the bipartite graph $\calI_{\sigma_1, s}$ between $\Gamma \times \{\sigma_1\}$ and $\Gamma \times \{s\}$ (recall \Cref{def:expanding-cubical-complex}).

\newcommand{\ahat}{\ol{a}}
\newcommand{\bhat}{\ol{b}}
\newcommand{\chat}{\ol{c}}

\begin{lemma} \label{lem:F(u;U)}
    Suppose that $u \in U_{\sigma_1}$ is such that $|N_{s}(u) \cap U| \le \nu$ for $s \in \{ \sigma_2, \sigma_3, \sigma_4 \}$. Then 
    \[
        |F_{C}(u;U;\sigma)| \le \nu^{3/2} \mper
    \]
\end{lemma}

\begin{proof}
    Let $a,b,c$ be the partition of $\Delta(\sigma) \subseteq [k]$ given by \Cref{claim:abc}. Define $A^{(a)} = \prod_{i \in a} A_i$, $A^{(b)} = \prod_{i \in b} A_i$, and $A^{(c)} = \prod_{i \in c} A_i$. There is a one-to-one correspondence between $N_{\sigma_2}(u)$ and $A^{(a)}A^{(b)} = A^{(b)}A^{(a)}$, $N_{\sigma_3}(u)$ and $A^{(a)}A^{(c)} = A^{(c)}A^{(a)}$, and $N_{\sigma_3}(u)$ and $A^{(b)}A^{(c)} = A^{(c)}A^{(b)}$.
    For instance, we can view $N_{\sigma_2}$ as the set of vertices obtained by starting from $u = (g_1, \sigma_1)$, and then multiplying $g_1$ first by an $A^{(a)}$ element and then an $A^{(b)}$ element to obtain a type-$\sigma_2$ vertex.

    By \Cref{lem:cube-count}, any $C$-face containing $u = (g_1,\sigma_1)$ can be uniquely specified by choosing one element each from $A^{(a)}$, $A^{(b)}$, and $A^{(c)}$. Concretely, \Cref{lem:cube-count} implies that for
    \(
        \ahat \in A^{(a)}, \bhat \in A^{(b)}, \chat \in A^{(c)},
    \)
    and $g_2 = g_1\,\ahat\,\bhat$, $g_3 = g_1\,\ahat\,\chat$,
    there is a unique $C$-face $f$ containing $(g_1,\sigma_1),(g_2,\sigma_2),(g_3,\sigma_3)$, where for $f_{\sigma_4}=(g_4,\sigma_4)$ we have $g_4 = g_1\,\bhat'\,\chat'$ for some $\bhat' \in A^{(b)}$ and $\chat'\in A^{(c)}$.
    Similarly, $f$ is also uniquely determined by the choice of $\ahat$, $\bhat'$, and $\chat'$.

    Let $H(\cdot)$ be the entropy function, and let $\boldf$ denote the random variable obtained by sampling a uniformly random $C$-face in $F_{C}(u;U;\sigma)$, and let $\ahat, \bhat, \chat, \bhat', \chat'$ denote the corresponding group elements. Then,
    \begin{align*}
        \log_2 |F_{C}(u;U;\sigma)|
        &= H\parens*{ \boldf } \\
        &= \frac12 \cdot H\big(\ahat, \bhat, \chat\big) + \frac12 \cdot H\big(\ahat, \bhat', \chat'\big) \\
        &= \frac12 \cdot \left( H\big(\ahat,\bhat\big) + H\big(\chat \mid \ahat, \bhat\big) \right)
        + \frac12 \cdot \left( H\left(\ahat\right) + H\big(\bhat', \chat' \mid \ahat\big) \right) \\
        &\le \frac12 \cdot \left( H\big(\ahat, \bhat\big) + H\left(\chat \mid \ahat\right) + H\left(\ahat\right) + H\big(\bhat', \chat'\big) \right) \\
        &= \frac12 \cdot \left( H\big(\ahat, \bhat\big) + H\big(\ahat, \chat\big) + H\big(\bhat', \chat'\big) \right) \\
        &\le \frac12 \cdot \left( \log_2 \lvert N_{s_2}(u) \cap U \rvert
        + \log_2 \lvert N_{s_3}(u) \cap U \rvert
        + \log_2 \lvert N_{s_4}(u) \cap U \rvert \right) \mcom
    \end{align*}
    or equivalently,
    \begin{align*}
        |F_{C}(u;U;\sigma)| &\le \sqrt{|N_{\sigma_2}(u) \cap U| \cdot |N_{\sigma_3}(u) \cap U| \cdot |N_{\sigma_4}(u) \cap U|} \mper \qedhere
    \end{align*}
\end{proof}

In \Cref{lem:F(u;U)}, we bounded the size of $F_{C}(u;U;\sigma)$ in terms of $\max_{s \in \{ \sigma_2, \sigma_3, \sigma_4 \}} |N_{s}(u) \cap U|$.
We also need to establish an upper bound on the number of $u \in U_{\sigma_1}$ with a given value of $\max_{s \in \{ \sigma_2, \sigma_3, \sigma_4 \}} |N_{s}(u) \cap U|$. 
To do this, we use the fact that our cubical complex $X$ is $2^k$-expanding, i.e., each bipartite graph $\calI_{\sigma_1, s}$ has second eigenvalue at most $2^k \sqrt{d_{\sigma_1 \oplus s}(X)} \leq 2^k \sqrt{d_{k/2}(X)}$ (\Cref{def:expanding-cubical-complex,thm:cubical-complexes}).
Here, we use that $\sigma_1 \oplus s$ has weight $k/2$ for $s\in \{\sigma_2,\sigma_3,\sigma_4\}$.

By our assumption that $\max_{i\in[k]} |A_i| \leq 2 \cdot \min_{i\in[k]}|A_i|$ and $D = \prod_{i\in[k]} |A_i|$, we have $d_{k/2} \coloneqq d_{k/2}(X) \le \sqrt{2^k D} = O_k(1) \cdot \sqrt{D}$.
For $1 \le \alpha \le 1+\log_2 d_{k/2}$, define 
\[
    U_{\sigma_1}(\alpha) \coloneqq \braces*{ u \in U_{\sigma_1} : \max_{s \in \{ \sigma_2, \sigma_3, \sigma_4 \}} |N_{s}(u) \cap U| \in [2^{\alpha-1}, 2^\alpha) } \mper
\]

\begin{lemma} \label{lem:bound-Ussprime}
    For any $\sigma \in \calH_k^4$ with $\sigma_1 \oplus \sigma_2 \oplus \sigma_3 \oplus \sigma_4 = 0$, it holds that 
    \[
        \abs*{U_{\sigma_1}(\alpha)} \le O_k(1) \cdot \min \left\{ 1,\ \frac{\sqrt{D}}{2^{2\alpha}} \right\} \cdot |U| \mper
    \]
\end{lemma}

\begin{proof}
    For $s \in \{ \sigma_2, \sigma_3, \sigma_4 \}$ and integer $\alpha \le 1 + \log d_{k/2}$, let us define 
    \[
        U_{\sigma_1,s}(\alpha) \coloneqq \braces*{ u \in U_{\sigma_1} : 2^{\alpha-1} \le |N_{s}(u) \cap U| < 2^\alpha } \mper
    \]
    Note that 
    \[
        |U_{\sigma_1}(\alpha)| \le \sum_{s \in \{ \sigma_2, \sigma_3, \sigma_4 \}} |U_{\sigma_1,s}(\alpha)| \mcom
    \]
    so it suffices to bound each $|U_{\sigma_1,s}(\alpha)|$ separately.

    To do this, we count the number of edges between $U_{\sigma_1,s}(\alpha)$ and $U_{s}$ in $\calI_{\sigma_1,s}$ in two different ways.
    First, by definition each $u \in U_{\sigma_1,s}(\alpha)$ has at least $2^{\alpha-1}$ neighbors within $U_{s}$, so we have that 
    \[
        |E(U_{\sigma_1,s}(\alpha), U_{s})| \ge 2^{\alpha-1} \cdot |U_{\sigma_1, s}(\alpha)| \mper \numberthis \label{eqn:edge-bound1}
    \]
    Second, by the expander mixing lemma on the graph $\calI_{\sigma_1, s}$ and using that $X$ is $2^k$-expanding and that $d_{k/2}$ is an upper bound on the degree of $\calI_{\sigma_1,s}$,
    \begin{align*}
        E(U_{\sigma_1,s}(\alpha), U_{s}) 
        &\le \frac{d_{k/2} \cdot |U_{\sigma_1,s}(\alpha)| \cdot |U_{s}|}{|\Gamma|} + 2^k \cdot \sqrt{d_{k/2}} \cdot \sqrt{|U_{\sigma_1,s}(\alpha)| \cdot |U_{s}|} \\
        &\le \left( d_{k/2} \cdot k D^{-1} + 2^k \cdot \sqrt{d_{k/2}} \right) \cdot \sqrt{|U_{\sigma_1,s}(\alpha)| \cdot |U_{s}|} \\
        &\leq O_k(1) \cdot D^{1/4} \cdot \sqrt{|U_{\sigma_1,s}(\alpha)| \cdot |U_{s}|} \mcom \numberthis \label{eqn:edge-bound2} 
    \end{align*}
    where in the second line we use that $|U| \leq D^{-1} \cdot |\Gamma\times \calH_k|$ and in the last line we use that $d_{k/2} = O_k(1) \cdot \sqrt{D}$.
    Combining \Cref{eqn:edge-bound1,eqn:edge-bound2}, this gives that 
    \[
        2^{\alpha - 1} \cdot |U_{\sigma_1,s}(\alpha)| \le O_k(1) \cdot D^{1/4} \cdot \sqrt{|U_{\sigma_1,s}(\alpha)| \cdot |U_s|} \mcom
    \]
    which rearranges to give
    \[
        |U_{\sigma_1,s}(\alpha)| 
        \le O_k(1) \cdot \frac{D^{1/2}}{2^{2\alpha}} \cdot |U_{s}| 
        \le O_k(1) \cdot \frac{D^{1/2}}{2^{2\alpha}} \cdot |U| \mper
    \]
    Thus,
    \[
        |U_{\sigma_1}(\alpha)| \le \sum_{s \in \{ \sigma_2,\sigma_3,\sigma_4 \}} |U_{\sigma_1,s}(\alpha)| \le O_k(1) \cdot \frac{D^{1/2}}{2^{2\alpha}} \cdot |U| \mper \numberthis \label{eqn:upper_U_s1}
    \]
    Finally, we obtain the lemma statement by combining~\Cref{eqn:upper_U_s1} with the fact that $|U_{s_1}(\alpha)| \le |U|$.
\end{proof}

We are now ready to prove \Cref{lem:small-set-subcube-density}.

\begin{proof}[Proof of \Cref{lem:small-set-subcube-density}]
    We first prove $|F_{k}(U;\sigma)| \le O_k(1) \cdot D^{5/8} \cdot |U|$ for $\sigma = (\sigma_1,\sigma_2,\sigma_3,\sigma_4) \in \calH_k^4$ that sums up to $0$ over $\F_2^k$.

    For the subcube $C = C_{\sigma} = \sigma_1 \oplus \bigoplus_{i\in \Delta(\sigma)} \{0,1\} \cdot e_i$ (\Cref{eq:subcube-sigma}), we can write
    \begin{align*}
        |F_{C}(U;\sigma)|
        &= \sum_{u \in U_{\sigma_1}} |F_{C}(u;U;\sigma)| \\
        &= \sum_{\alpha = 1}^{1+\log d_{k/2}} \sum_{u \in U_{\sigma_1}(\alpha)} |F_{C}(u;U;\sigma)| \\
        &\le \sum_{\alpha = 1}^{1+\log d_{k/2}} |U_{\sigma_1}(\alpha)| \cdot 2^{3\alpha/2} \\
        &\le \sum_{\alpha = 1}^{1+\log d_{k/2}} O_k(1) \cdot \min \left\{ 1, \frac{D^{1/2}}{2^{2\alpha}} \right\} \cdot |U| \cdot 2^{3\alpha/2} \\
        &= O_k(1) \sum_{\alpha = 1}^{(\log D)/4}  2^{3\alpha/2} \cdot |U| + O_k(1) \sum_{\alpha = 1+(\log D)/4}^{1+\log d_{k/2}} \frac{D^{1/2}}{2^{\alpha/2}} \cdot |U| \\
        &\leq O_k(1) \cdot D^{3/8} \cdot |U| \mcom
    \end{align*}
    where the first inequality follows from \Cref{lem:F(u;U)} (since every $u \in U_{\sigma_1}(\alpha)$ satisfies $|N_s(u) \cap U| \leq 2^{\alpha}$ for $s\in \{\sigma_2,\sigma_3,\sigma_4\}$ by definition), and the second inequality follows from \Cref{lem:bound-Ussprime}.

    Next, by~\Cref{lem:cube-count}, each $f \in F_{C}(U; \sigma)$ can be extended to $f \in F_k(U; \sigma)$ in $\prod_{i \not\in \Delta(\sigma)} |A_i| \le O_k(1) \cdot D^{1/4}$ ways, so
    \begin{align*}
        |F_k(U;\sigma)| 
        \le |F_{C}(U;\sigma)| \cdot O_k(1) \cdot D^{1/4} 
        \le O_k(1) \cdot D^{5/8} \cdot |U| \mper
    \end{align*}
    Finally, by plugging in the above into \Cref{eq:final-bound}, we obtain the desired inequality:
    \begin{align*}
        |F_k(U; \ge 2\sqrt{k})| \le O_{k}(1) \cdot D^{5/8} \cdot |U| \mper &  \qedhere
    \end{align*}
\end{proof}

%% file: LPS.tex
\section{Ramanujan cubical complexes}
\label{sec:cubical-construction}

% Defining these here instead of the macros
\newcommand{\I}{\mathbf{{i}}}
\newcommand{\J}{\mathbf{{j}}}
\newcommand{\K}{\mathbf{{k}}}
\newcommand{\Quat}{\calH(\Z)}

In this section, we give a proof of \Cref{thm:cubical-complexes}, which is essentially contained in \cite{RSV19}.
In particular, we describe the construction of expanding cubical complexes (\Cref{def:expanding-cubical-complex}) based on the LPS Ramanujan graphs \cite{LPS88}.
For our purposes, we only need basic properties of the generating sets of these Cayley graphs, while using the (highly non-trivial) fact that they are Ramanujan as a black box.

\subsection{LPS Ramanujan graphs}

In this section, we give a brief overview of the LPS Ramanujan graphs \cite{LPS88} (see also \cite{Lubotzky94}).

\parhead{Notation.}
For any $n \in \N$, let $r_4(n) \coloneqq |\{(a,b,c,d) \in \Z^4: a^2 + b^2 + c^2 + d^2 = n\}|$.

We start with a standard fact.
\begin{fact}[Jacobi's four-square theorem]
\label{fact:jacobi-sos}
    For any odd $n$, $r_4(n) = 8\sum_{m|n} m$.
    In particular, if $n = p_1 p_2 \cdots p_k$ for distinct odd primes $p_1,\dots,p_k$, then $r_4(n) = 8 \prod_{i=1}^k (p_i+1)$.
\end{fact}

Let us start with the definition of quaternions.
We will restrict our attention to integral quaternions (a.k.a.\ Lipschitz quaternions).

\begin{definition}[Integral quaternions]
\label{def:quaternions}
    Define $\Quat = \{a\,\id + b\I + c\J + d \K: a,b,c,d\in \Z\}$ where
    \begin{align*}
        \I =
        \begin{bmatrix}
        i & 0 \\
        0 & -i
        \end{bmatrix} \mcom
        \qquad
        \J =
        \begin{bmatrix}
        0 & 1 \\
        -1 & 0
        \end{bmatrix} \mcom
        \qquad
        \K =
        \begin{bmatrix}
        0 & i \\
        i & 0
        \end{bmatrix} \quad
        \in \C^{2\times2} \mper
    \end{align*}
    For $\alpha = a\id + b\I + c\J + d\K \in \Quat$, we define its norm $N(\alpha)$ as $\det(\alpha) = a^2+b^2+c^2+d^2$, and we define the (normalized) trace $\tr(\alpha) = a$.
\end{definition}

\begin{remark}
    It can be verified that $\I$, $\J$, $\K$ in \Cref{def:quaternions} satisfy the following relations:
    \[
        \I^2 = \J^2 = \K^2 = \I\J\K = -\id \mper
    \]
    The quaternions are traditionally defined according to these relations.
    \Cref{def:quaternions} is a \emph{matrix representation} of quaternions in $\C^{2\times2}$.
\end{remark}

Note that the norm is a multiplicative map: $N(\alpha \beta) = \det(\alpha 
\beta) = N(\alpha) N(\beta)$.
Thus, for integral quaternions, the group of units is
\begin{align*}
    \Quat^{\times} = \{\pm\id, \pm \I, \pm\J, \pm\K\} \mper
\end{align*}

We now formulate the ``unique factorization'' theorem for $\Quat$.
% , which follows from Theorem 2 and Lemma 1 of \cite[Chapter 5]{CS03}.
This is a key property that we will need later to construct the Ramanujan cubical complexes (see \Cref{sec:cubical-construction}).

% $\Quat$ is non-commutative and does not satisfy the unique factorization property.
% Fortunately, it satisfies a weak version of unique factorization \cite{Dickson22} (see also \cite{Pall40,CS03,Tso20} for expositions).\footnote{One can extend $\Quat$ to the \emph{Hurwitz} quaternions with more desirable properties. We refer readers to \cite{CS03,Tso20} for more thorough expositions.}

\begin{fact}[{Unique factorization~\cite[Theorem 8]{Dickson22}}]
\label{fact:unique-factorization}
    Let $\alpha \in \Quat$ such that $N(\alpha)$ is odd.\footnote{$N(\alpha)$ being odd is necessary because $2 = (1+\I)(1-\I) = (1+\J)(1-\J)$, which is not unique up to unit migration. One can extend $\Quat$ to the \emph{Hurwitz} quaternions to handle this case (see, e.g., \cite{Pall40,CS03}).}
    Let $N(\alpha) = p_1 p_2 \cdots p_k$ be the factorization of the norm into primes, arranged in an arbitrary but definite order.
    Then, there is a decomposition $\alpha = \alpha_1 \alpha_2 \cdots \alpha_k$ where $N(\alpha_i) = p_i$ for each $i\in[k]$.
    Moreover, the decomposition is unique up to ``unit migration'', where $\alpha_1 \alpha_2 \cdots \alpha_k$ and $(\alpha_1 u_1) (\ol{u}_1 \alpha_2 \ol{u}_2) \cdots (\ol{u}_{k-1} \alpha_k)$ for any $u_1,\dots, u_{k-1} \in \Quat^{\times}$ are considered the same decomposition.
    % , and this decomposition is unique up to association of the factors.
    % Let $\alpha = a\id + b\I + c\J + d\K \in \Quat$ with odd norm, and let $p\in \N$ be an odd integer such that $p|N(\alpha)$ and $\gcd(a,b,c,d,p) = 1$.\footnote{$N(\alpha)$ being odd is necessary because $2 = (1+\I)(1-\I) = (1+\J)(1-\J)$, which is not unqiue up to association. One can extend $\Quat$ to the \emph{Hurwitz} quaternions to handle this case (see, e.g., \cite{Pall40,CS03}).}
    % Then, there is a unique, up to left multiplication by units, right divisor of $\alpha$ of norm $p$.
    % That is, there is a unique set $B = \Quat^{\otimes} \cdot \beta \subseteq \Quat$ such that for each $\beta'\in B$, $N(\beta') = p$ and $\alpha = \gamma \beta'$ for some $\gamma \in \Quat$.
\end{fact}

Note that factorization can only be unique up to unit migration simply because $\alpha \beta = (\alpha \ol{u}) (u\beta)$ for any unit $u\in \Quat^{\times}$.\footnote{This is similar for integers $\Z$ where factorization is unique up to the association $a \sim -a$.}

Next, we define the following, which will later give us the generators of the LPS graphs.

\begin{definition} \label{def:generators}
    For $n\in \N$, define
    \begin{align*}
        A(n) \coloneqq \{\alpha \in \Quat: N(\alpha) = n,\ \tr(\alpha) \text{ is odd}\}\ /\ \{\id, -\id\} \mper
    \end{align*}
\end{definition}

It is convenient to view this quotient as the set of odd-trace quaternions where $\alpha$ and $-\alpha$ are considered to be identical.

The following fact is a simple consequence of Jacobi's four-square theorem (\Cref{fact:jacobi-sos}). We will prove a generalization later (\Cref{lem:LPS-cubical-generating}).

\begin{fact}
    For a prime $p$ congruent to $1$ modulo $4$, $|A(p)| = p+1$.
\end{fact}

\parhead{LPS Ramanujan graphs.}
We now describe the LPS Ramanujan graphs $X(p;q)$, where
\begin{itemize}
    \item $p < q$ are primes congruent to $1$ modulo $4$,
    \item $p$ is a quadratic residue modulo $q$ --- that is, there exists $x\in \Z$ such that $p \equiv x^2 \pmod{q}$.\footnote{\cite{LPS88} also defined Cayley graphs when $p$ is \emph{not} a quadratic residue. In this case, the graphs are over $\PGL(2,\F_q)$ and they are bipartite. We will not consider this case.}
\end{itemize}
The graph is a Cayley graph over the group $\PSL(2, \F_q)$ with $p+1$ generators defined by $A(p)$ (\Cref{def:generators}).
Here, $\PSL(2, \F_q)$ is the \emph{projective special linear group}: it is a subgroup of $2\times2$ matrices in $\F_q$ of determinant $1$ modulo scalar multiplication, i.e., $\wt{\alpha}$ belongs to the equivalence class $[c \wt{\alpha}]$ if $\det(c \wt{\alpha}) = c^2 \det(\wt{\alpha}) = 1$ (in $\F_q$).
It is easy to check that $|\PSL(2,\F_q)| = q(q^2-1)/2$.

We first need to map a quaternion $\alpha\in A(p)$ to an element in $\PSL(2, \F_q)$.
To do so, we need an element $j \in \F_q$ such that $j^2 = -1$ (thus behaving like the imaginary unit $i$).
This requires $q \equiv 1 \pmod{4}$, in which case it is well known (by Euler's criterion) that $-1$ is a quadratic residue mod $q$, i.e., there exists $y\in \Z$ such that $y^2 \equiv -1 \pmod{q}$.

Moreover, each $\alpha \in A(p)$ has $\det(\alpha) = p$.
We need that there exists $c \in \Z$ such that $\det(c\alpha) = c^2 p \equiv 1 \pmod{q}$ to get an element in $\PSL(2,\F_q)$.
Thus, choosing $p$ such that $p \equiv x^2 \pmod{q}$ for some $x \in \Z$, since there always exists $c\in \Z$ such that $cx \equiv 1 \pmod{q}$, we have that $c^2 p \equiv c^2 x^2 \equiv 1 \pmod{q}$.

This gives a natural map $\alpha \in A(p)$ to $\wt{\alpha} \in \PSL(2,\F_q)$ by simply replacing $i$ with $j\in \F_q$ with $j^2 = -1$.
 We denote
 \begin{align*}
     \wt{A}(p) \coloneqq \{\wt{\alpha}: \alpha\in A(p)\} \mper
 \end{align*}
 Note that $|\wt{A}(p)| = |A(p)| = p+1$, since no distinct $\alpha,\beta\in A(p)$ are scalar multiples of each other.

 The following is the main theorem of \cite{LPS88} whose proof is out of the scope of this paper.

\begin{theorem}[\cite{LPS88}] \label{thm:ramanujan}
    Suppose $p < q$ are primes congruent to $1$ modulo $4$, and $p$ is a quadratic residue modulo $q$.
    Let $\Gamma = \PSL(2,\F_q)$.
    Then, the Cayley graph $\Cay(\Gamma; \wt{A}(p))$ is a $(p+1)$-regular graph on $q(q^2-1)/2$ vertices with all non-trivial eigenvalues at most $2\sqrt{p}$.
\end{theorem}

\subsection{Construction of Ramanujan Cayley cubical complexes}

The following is an important lemma that allows us to construct cubical complexes.
The proof is straightforward given \Cref{fact:jacobi-sos,fact:unique-factorization}.

\begin{lemma} \label{lem:LPS-cubical-generating}
    For any $k\in \N$ and distinct primes $p_1, p_2,\dots, p_k$ congruent to $1$ modulo $4$,
    \begin{enumerate}[(1)]
        \item $|A(p_1p_2 \cdots p_k)| = \prod_{i=1}^k (p_i+1)$.
        \label{item:size-of-A}
        \item $A(p_1) \cdot A(p_2) \cdots A(p_k) = A(p_1p_2 \cdots p_k)$.
        \label{item:set-commute}
    \end{enumerate}
\end{lemma}
\begin{proof}
    First, note that any number $x$ has $x^2 \equiv 1 \pmod{4}$ if $x$ is odd, and $0$ otherwise.
    Thus, $p \equiv 1 \pmod{4}$ implies that for $a_0^2 + a_1^2 + a_2^2 + a_3^2 = p$, the set $a_0,a_1,a_2,a_3$ must have exactly one odd and three even integers.
    Note also that $p_i \equiv 1 \pmod{4}$ implies that $p_1 p_2 \cdots p_k \equiv 1 \pmod{4}$.

    With a slight abuse of notation, we will view an element $\alpha$ of $A(n)$ as a quaternion even though it is technically a coset $\{\alpha, -\alpha\}$, since $N(\alpha) = N(-\alpha)$ and $\tr(\alpha), \tr(-\alpha)$ have the same parity.

    For \ref{item:size-of-A}, let $n = p_1 p_2 \cdots p_k$.
    By Jacobi's four-square theorem (\Cref{fact:jacobi-sos}), $r_4(n) = 8 \prod_{i=1}^k (p_i+1)$.
    Since $A(n)$ has the restriction that $\tr(\alpha)$ is odd, each element in $A(n)$ gives rise to $8$ distinct $4$-tuples of integers whose squares sum up to $n$ (by specifying the position of the odd integer and its sign).
    This shows that $|A(n)| = \frac{1}{8} r_4(n) = \prod_{i=1}^k (p_i+1)$.

    For \ref{item:set-commute}, we first show that for any $n_1 \neq n_2$ congruent to $1$ modulo $4$, we have $A(n_1) \cdot A(n_2) \subseteq A(n_1n_2)$.
    This implies that $A(p_1) \cdot A(p_2) \cdots A(p_k) \subseteq A(p_1 p_2 \cdots p_k)$ as all $p_i \equiv 1 \pmod{4}$.
    For any $\alpha = a_0 \id + a_1 \I + a_2 \J + a_3 \K \in A(n_1)$ and $\beta = b_0 \id + b_1 \I + b_2 \J + b_3 \K \in A(n_2)$, we have that $N(\alpha \beta) = N(\alpha) N(\beta) = n_1 n_2$.
    Moreover, we know that $a_0, b_0$ are odd and the rest are even, thus $\tr(\alpha\beta) = a_0 b_0 - a_1 b_1 - a_2 b_2 - a_3 b_3$ is odd.
    This implies that $\alpha\beta \in A(n_1 n_2)$.

    On the other hand, $A(p_1 p_2 \cdots p_k) \subseteq A(p_1) \cdot A(p_2) \cdots A(p_k)$ follows directly from unique factorization (\Cref{fact:unique-factorization}).
    % Next, we show that if $p$ is a prime and $p,n$ are coprime and congruent to $1$ modulo $4$, then $A(n p) \subseteq A(n) \cdot A(p)$.
    % This, combined with the above, proves \ref{item:set-commute}.
    % For any $\gamma = c_0 \id + c_1\I + c_2\J + c_3\K \in A(np)$, since $n$ and $p$ are coprime and $c_0^2 + c_1^2 + c_2^2 + c_3^2 = np$, we must have $\gcd(c_0, c_1,c_2,c_3, p) = 1$.
    % By \Cref{fact:unique-factorization}, there exists $\beta = b_0\id + b_1\I + b_2\J + b_3\K \in \Quat$ with $N(\beta) = p$ such that $\gamma = \alpha \beta$ for some $\alpha \in \Quat$, and $\beta$ is unique up to multiplication by the units $\{\pm \id, \pm\I, \pm\J, \pm\K\}$.
    % Since $p \equiv 1 \pmod{4}$, the set $\{b_0,b_1,b_2,b_3\}$ has exactly one odd integer, thus there is a unique $\beta \in A(p)$ (more specifically, $\{\pm\beta\}$) with $b_0$ odd such that $\gamma = \alpha\beta$ for some $\alpha = a_0\id + a_1\I + a_2\J + a_3\K\in \Quat$.
    % The norm of $\alpha$ must be $n$ since $N(\alpha)N(\beta) = np$.
    % Moreover, $\tr(\alpha \beta) = a_0 b_0 - a_1 b_1 - a_2 b_2 - a_3 b_3$ is odd implies that $a_0$ is odd, so $\alpha\in A(n)$.
\end{proof}

The next lemma follows almost immediately from \Cref{thm:ramanujan,lem:LPS-cubical-generating}.

\begin{lemma}
    Let $p_1, p_2, \dots, p_k$ and $q$ be distinct primes congruent to $1$ modulo $4$, and suppose each $p_i$ is a quadratic residue modulo $q$.
    Let $\Gamma = \PSL(2,\F_q)$.
    Consider the bipartite graph $G$ defined on $\Gamma \times \zo$ where $(g,0)$ and $(h, 1)$ are connected if and only if $g^{-1} h \in \wt{A}(p_1 p_2 \cdots p_k)$.
    Then, $G$ has degree $d = \prod_{i=1}^k |\wt{A}(p_i)| = \prod_{i=1}^k (p_i+1)$ and second eigenvalue at most $2^k \sqrt{d}$.
\end{lemma}
\begin{proof}
    By \Cref{lem:LPS-cubical-generating}, we have that $A(p_1) \cdot A(p_2) \cdots A(p_k) = A(p_1 p_2 \cdots p_k)$ and that $|A(p_1p_2 \cdots p_k)| = \prod_{i=1}^k (p_i+1)$.
    Thus, the degree $d = \prod_{i=1}^k (p_i+1)$.
    The adjacency matrix of $G$ is the (bipartite form of) product of adjacency matrices of $\Cay(\Gamma; \wt{A}(p_i))$.
    The trivial eigenvector is the all-ones vector for all these graphs, and thus, by submultiplicativity of the spectral norm, the second eigenvalue of $G$ is at most the product of the second eigenvalues of $\Cay(\Gamma; \wt{A}(p_i))$, which is $\prod_{i=1}^k (2\sqrt{p_i}) \leq 2^k \sqrt{d}$ by \Cref{thm:ramanujan}.
\end{proof}

\parhead{Infinite family of cubical complexes.}
For any distinct primes $p_1,p_2,\dots,p_k$, we need to show that there are infinitely many desirable primes $q$: congruent to $1$ modulo $4$ and that each $p_i$ is a quadratic residue modulo $q$.
This is standard and follows directly from the law of quadratic reciprocity and the Dirichlet prime number theorem.

\begin{lemma}
    Let $p_1, p_2,\dots, p_k$ be distinct primes congruent to $1$ modulo $4$.
    There are infinitely many primes $q$ such that $q\equiv 1 \pmod{4}$ and that each $p_i$ is a quadratic residue modulo $q$.
\end{lemma}
\begin{proof}
    Let $n = p_1 p_2 \cdots p_k$, and consider the arithmetic progression $\{1+ 4n \ell\}_{\ell \in \N}$.
    The Dirichlet prime number theorem states that this sequence contains infinitely many prime numbers (since $1$ and $4n$ are coprime).
    For any such prime $q$, we have $q\equiv 1\pmod{4}$ and $q \equiv 1 \pmod{p_i}$ for each $i$, which also means that $q$ is a quadratic residue modulo $p_i$.
    Then, quadratic reciprocity implies that each $p_i$ is a quadratic residue modulo $q$.
\end{proof}

We also need to argue that there exist such primes that are all within a constant factor apart.
This follows from standard facts about the density of primes in arithmetic progressions (see e.g., \cite{BMOR18}).

\begin{fact} \label{fact:prime-density}
    For any $k \in \N$ and $B > 1$, there exists $x_0 = x_0(k,B)$ such that for any $x \ge x_0$, there are distinct primes $p_1,p_2,\dots,p_k \in [x, Bx]$ congruent to $1$ modulo $4$.
\end{fact}

%% file: acknowledgements.tex
\section*{Acknowledgements}
R.Y.Z.\ would like to thank Mehtaab Sawhney for his extensive knowledge of combinatorics and helpful references within.
S.M.\ would like to thank Ryan O'Donnell for helpful conversations about expanding Cayley graphs.
J.H.\ would like to thank Mitali Bafna for helpful discussions.

%% file: free-action.tex
\section{Free group action and good quantum LDPC codes}
\label{app:free-action}

The main result of~\cite{LH22b} is a construction of good quantum low density parity check (qLDPC) codes with a linear time decoding algorithm, assuming the existence of two-sided lossless expanders with a free group action, which they left as a conjecture. We state their conjecture below.

\begin{conjecture}[\cite{LH22b}, Conjecture 10] \label{conj:LH22}
    For any $\epsilon > 0$, and for any $\beta \in (0, 1]$ and $\epsilon_0 > 0$, there are $d_L, d_R \in \bbN$ satisfying $\frac{d_R}{d_L} \in [\beta, \beta + \epsilon_0]$, a constant $\eta > 0$, and an infinite family of $(d_L, d_R)$-biregular bipartite graphs $\{ Z_i = (L_i, R_i, E_i) \}$ and groups $\{ G_i \}$, satisfying the following properties:
    \begin{enumerate}[label=(\Roman*)]
        \item \label{item:two-sided-lossless} $Z_i$ is a two-sided $(1-\epsilon)$-vertex expander. Namely, any $S \subseteq L_i$ with $|S| \le \eta \cdot |L_i|$ has $\ge (1-\epsilon) d_L \cdot |S|$ neighbors on the right, and any $S \subseteq R_i$ with $|S| \le \eta \cdot |R_i|$ has $\ge (1-\epsilon) d_R \cdot |S|$ neighbors on the left.
        \item \label{item:free-group-action} $|G_i| = O(|Z_i|)$, and $Z_i$ has a free $G_i$-action.
    \end{enumerate}
\end{conjecture}

\noindent
Lin and M. Hsieh used such two-sided lossless expanders to construct good qLDPC codes. 

\begin{theorem}[\cite{LH22b}, Theorem 9 and Theorem 14] \label{thm:LH22}
    Assuming~\Cref{conj:LH22}, then for all $r \in (0, 1)$, there exists $\delta > 0$, $w \in \bbN$ and a infinite family of quantum error-correcting codes $C = \{ C_i \}_{i \in \bbN}$ with parameters $[[n_i, k_i, d_i]]$, such that $k_i/n_i > r$, $d_i/n_i > \delta$, and all stabilizers of $C_i$ have weight $w$. Furthermore, $C$ has a linear time decoding algorithm.
\end{theorem}

In what follows, we show that the graphs we construct in~\Cref{sec:construction} resolve~\Cref{conj:LH22}, thereby giving a new instantiation of qLDPC codes via the framework of~\cite{LH22b}.
We have already proved Condition~\ref{item:two-sided-lossless} in~\Cref{thm:main}.
It remains simply to check that the groups $G_i$ satisfying Condition~\ref{item:free-group-action} exist.

\begin{proposition}
    The graph $Z$ constructed in~\Cref{sec:construction}, using Cayley cubical complexes over $\Gamma = \PSL(2, \bbF_q)$, has a free $\Gamma$-action.
\end{proposition}

\begin{proof}
    We begin by recalling some notation. Let $X = \Cay(\Gamma, \calA)$ be a cubical complex over $\Gamma$, where $\calA = \{ A_1, \dots, A_k \}$ are $k$ sets of Cayley cubical generators. The graphs $G_L = (L, M, E_L)$ and $G_R = (M, R, E_R)$ are defined as follows:
    \begin{itemize}
    \item $L = \{ v \in X(k) : \mathrm{Signature}(v) \in \calS_L \}$, where $\calS_L \subseteq \ul{\calS}_L$ is a $D_L$-sized collection of signatures,
    \item $R = \{ v \in X(k) : \mathrm{Signature}(v) \in \calS_R \}$, where $\calS_R \subseteq \ul{\calS}_R$ is a $D_R$-sized collection of signatures (see~\Cref{sec:proof-lem-base-graph}),
    \item $M = \Gamma \times \calH_k$,
    \item $(f, u) \in E_L$ if $u \in f$, and $(u, f) \in E_R$ if $u \in f$.
    \end{itemize}
    Then, the graph $Z$ was constructed by placing a copy of the gadget graph $H$ on the left and right neighbors of each $u \in M$. Precisely, for each edge $(i,j) \in H$, we place an edge between $\LNbr_u(i)$ and $\RNbr_u(j)$. 

    We claim that $Z$ has a free left $\Gamma$-action. This will essentially follow from the observations that $G_L$ and $G_R$ permit a free left $\Gamma$-action, and the placement of the gadget $H$ respects the group structure. 
    
    More concretely, let us define the left $\Gamma$-action on $u = (g, x) \in M$ as follows: 
    \[
        \gamma u := (\gamma g, x).
    \]
    We can also define a left $\Gamma$-action on $\ul{L} = \ul{R} = X(k)$: for $f = \{ (f_x, x) \}_{x \in \{ 0, 1 \}^k}$, we define
    \[
        \gamma f := \{ (\gamma f_x, x) \}_{x \in \calH_k}.
    \]
    It turns out that because the cubical generating sets $A_i$ all act on the right, this defines a legal action on $X(k)$ as well, which we check by verifying $\gamma f \in X(k)$:
    \[
        (\gamma f)_x^{-1} (\gamma f)_{x + e_i} = f_x^{-1} \gamma^{-1} \gamma f_{x + e_i} = f_x^{-1} f_{x+e_i} \in A_i. \numberthis \label{eqn:cubical-group-action}
    \]
    Both the above actions are free because $\Gamma$ acting on itself is free. This will imply that the left $\Gamma$-action on $Z$, which has vertex set a subset of $X(k)$, is free as well.

    \Cref{eqn:cubical-group-action} actually implies something even stronger: acting on the left by $\gamma$ preserves the signature of the cube. It follows that the subsets $L \subseteq \ul{L}$ and $R \subseteq \ul{R}$ also permit a free left $\Gamma$-action, since $L$ and $R$ consist of all cubes with a certain collection of signatures. Now looking at the base graph $G_L$, we define for $(f, u) \in E_L$
    \[
        \gamma (f, u) := (\gamma f, \gamma u).
    \]
    This defines a valid left $\Gamma$-action on $E_L$, since if $u \in f$ then $\gamma u \in \gamma f$. Similarly, we can define for $(u, f) \in E_R$ 
    \[
        \gamma (u, f) := (\gamma u, \gamma f).
    \]
    Note in particular that if $f$ is the neighbor of $u$ with a given signature $\sigma$, then $\gamma f$ is the neighbor of $\gamma u$ with signature $\sigma$.

    Next, we show that the placement of the gadget graph $H$ respects the left $\Gamma$ action. Recall that in~\Cref{sec:proof-lem-base-graph}, $\LNbr_u$ (similarly, $\RNbr_u$) were defined so that $\Signature(\LNbr_u(i)) = \Signature(\LNbr_{u'}(i))$ for any $u, u' \in \Gamma \times \{ \sigma \}$, $\sigma \in \calH_k$. From the above discussion, this implies that 
    \[
        \gamma \LNbr_u(i) = \LNbr_{\gamma u}(i).
    \]
    In particular, under a left $\Gamma$-action, an edge $(\LNbr_u(i), \RNbr_u(j)) \in E$ gets sent to
    \[
        \gamma (\LNbr_u(i), \RNbr_u(j)) := (\gamma \LNbr_u(i), \gamma \RNbr_u(j)) 
        = (\LNbr_{\gamma u}(i), \RNbr_{\gamma u}(j)) 
        \in E.
    \]

    Finally, we check that $\Gamma$ has linear size:
    \[
        |Z| \le 2 |X(k)| = 2|\Gamma| \cdot \prod_{i=1}^k |A_i| = O_k(1) \cdot |\Gamma|. \qedhere
    \]
\end{proof}